\documentclass[a4paper,8pt,oneside,bookmarks]{article}
\usepackage[utf8x]{inputenc}
\usepackage{graphicx}

\usepackage{wrapfig}
\usepackage[pdftex,bookmarks]{hyperref}
\hypersetup{
pdftitle={On conditional weak-star compactness},
pdfauthor={José M. Zapata}
}

\usepackage[english]{babel}
\usepackage{amsmath}
\usepackage{amsthm}
\usepackage{amsfonts}

\usepackage{amscd}

\usepackage{epigraph}

\usepackage{amssymb}

\usepackage{relsize}

\newcommand{\N}{\mathbb{N}}
\newcommand{\R}{\mathbb{R}}
\newcommand{\E}{\mathbb{E}}

\newcommand{\Q}{\mathbb{Q}}
\newcommand{\F}{\mathcal{F}}
\newcommand{\A}{\mathcal{A}}

\DeclareMathOperator{\spa}{\textbf{span}}

\DeclareMathOperator{\supp}{supp}
\DeclareMathOperator{\nmax}{\textbf{max}\,}
\DeclareMathOperator{\nmin}{\textbf{min}\,}
\DeclareMathOperator{\ninf}{\textbf{inf}\,}
\DeclareMathOperator{\nsup}{\textbf{sup}\,}
\DeclareMathOperator{\nlim}{\textbf{lim}\,}
\DeclareMathOperator{\conprod}{\mathlarger{\mathlarger{\mathlarger{\Join}}}}

\newtheorem{defn}{Definition}[section]
\newtheorem{rem}{Remark}[section]
\newtheorem{prop}{Proposition}[section]
\newtheorem{lem}{Lemma}[section]
\newtheorem{thm}{Theorem}[section]

\begin{document}


\oddsidemargin 16.5mm
\evensidemargin 16.5mm

\thispagestyle{plain}

%
%

\vspace{5cc}
\begin{center}

{\large\bf  VERSIONS OF EBERLEIN-\v{S}MULIAN AND AMIR-LINDENSTRAUSS THEOREMS IN THE FRAMEWORK OF CONDITIONAL SETS
\rule{0mm}{6mm}\renewcommand{\thefootnote}{}
\footnotetext{

\rule{2.4mm}{0mm}Keywords and Phrases. conditional set theory; conditional weak topology; Amir-Lindenstrauss theorem; Eberlein-\v{S}mulian theorem; Baire Category theorem.
}}

\vspace{1cc}
{\large\it Jos\'e M. Zapata}

\vspace{1cc}
\parbox{24cc}{{\small

Based on conditional set theory, we study conditional weak topologies, extending some well-known results to this  framework and culminating with the proof of conditional versions of Eberlein-\v{S}mulian and Amir-Lindenstrauss theorems. In pursuing this aim, we prove conditional versions of Baire Category theorem and Uniform Boundedness Principle.

%

}}
\end{center}

\vspace{1cc}

\vspace{1.5cc}
\begin{center}
{\bf INTRODUCTION}
\end{center}

This paper is a sequel of the manuscript \cite{key-7}. There, S. Drapeau, A. Jamneshan, M. Karliczek and M. Kupper succeeded in creating a new formal language together with constructive methods, which provide an alternative to classical measurable selection arguments. On this basis, they managed to draw many basic results from real analysis, employing a formulation that consistently depends on the elements of a measure algebra associated to a $\sigma$-finite measure space, or more in general, of a complete Boolean algebra. This method allows to deal with dynamic settings in which there exists a consistent flow of information. This type of problems frequently arise in mathematical finance. For instance, dynamic measurement of risk has gained increasing importance in the recent literature cf.\cite{key-32,key-33,key-20,key-21,key-22}. This is because, unlike the static case, it gives a precise and consistent measurement of the risk of financial positions over time, taking into account the arrival of new information throughout the strategy of risk management.


In mathematical finance, many results, specifically robust representation results of risk measures or monetary utility functions (i.e. the negative value of a risk measure), rely on weak and weak-$*$ compactness theorems cf.\cite{key-25,key-24,key-26}. In particular, we find that those results which link weak or weak-$*$ compactness to convergence of sequences are extremely useful, since weak and weak-$*$ topologies generally lack a countable neighborhood base of $0$. For example, a theorem in this direction, which is considered one of deepest theorems in the study of weak topologies, is the classical Eberlein-\v{S}mulian theorem, which states that a subset of a normed space is weakly compact if, and only if, it is weakly sequentially compact. Regarding the weak-$*$ topology, another deep result is the so-called Amir-Lindenstrauss theorem. This result states that, for a weakly compactly generated Banach space $E$ (i.e. $E$ is spanned by a weakly-$*$ compact, convex and balanced subset), its dual unit ball $B_{E^*}$ turns out to be weakly-$*$ sequentially compact.  

It is for this reason that results which relate weak or weak-$*$ compactness to convergence of sequences have applications at the core of the theory of representation of risk measures or monetary utility functions. For instance, the so-called Jouini-Schachermayer-Touzi theorem (see \cite[Theorem 24]{key-25}) is a robust representation theorem for convex risk measures defined on $L^\infty$, the set of (classes of equivalence) of bounded measurable functions in some underlying probability space $(\Omega,\mathcal{F},P)$. This theorem states that a convex risk measure $\rho:L^\infty\rightarrow \R$, with some mild continuity condition called Fatou property, satisfies that for each $x\in L^\infty$ there is a probability $Q$ such that $\rho(x)=\E_Q[-x]+c(Q)$ if, and only if, the sublevel set $V_k=\left\{ Q \:;\: c(Q)\leq k \right\}$ is weakly compact in $L^1$ (with the identification of the probability measure and the Radon-Nikodym derivative $d Q/d P\in L^1$). Then, Eberlein-\v{S}mulian theorem is used for proving the sufficiency of the last condition; in this case, a sequence contained in the weakly compact subset $V_k$ is considered, and then a convergent subsequence is taken. 
In \cite{key-35} the Eberlein-\v{S}mulian theorem is used to prove that, for a convex law determined risk measure $\rho:L^\infty\rightarrow\R$, a weak form of mixture continuity is equivalent to robustness (see \cite[Proposition 2.7]{key-35}).  Another example is \cite{key-26}, where a perturbed version of the classical weak compactness James' theorem is proved with the aim of providing a generalization of Jouini-Schachermayer-Touzi theorem for Orlicz spaces. In this work, Eberlein-\v{S}mulian has a key role (see for instance proofs of Theorem 1 and Lemma 4 of \cite{key-26}).

Thereby, the aim of this paper is to provide a solid study of weak and weak-$*$ compactness in a conditional setting, which can give rise to results in the study of dynamic risk measures. We study some basic results of functional analysis under this approach. For instance, conditional versions of Baire Category theorem and Uniform  Boundedness Principle are provided, as well as conditional versions of  Goldstine's theorem and other basic results on weak topologies. We finally culminate with the proofs of the main results of this paper.   Namely, we will test that the statements of the classical Eberlein-\v{S}mulian and Amir-Lindenstrauss theorems  naturally extend to the conditional setting. Being more concrete, the conditional version of Eberlein-\v{S}mulian theorem states that a conditional subset of a conditionally  normed space is conditionally weakly compact if, and only if, it is conditionally weakly sequentially compact, and the conditional version of Amir-Lindenstrauss theorem claims that, if a conditional Banach space is conditionally spanned by a conditionally weakly-$*$ compact, conditionally convex and conditionally balanced subset, then its conditional dual unit ball is conditionally weakly-$*$ sequentially compact.  
  
The paper is structured as follows: in Section 1 we first provide some preliminaries, recall the main notions of the conditional set theory, and fix the notation of the paper. Section 2 is devoted to prove some basic results of functional analysis. Finally, in Section 3 we will prove the conditional versions of Eberlein-\v{S}mulian and Amir-Lindenstrauss theorems.


\section{Preliminaries and notation}   
\label{sec1}


Let $(\mathcal{A},\vee,\wedge,\empty^c,0,1)$ be a complete Boolean algebra, which will be the same throughout this paper. For instance, we can think of the complete Boolean algebra $\A$ $-$called measure algebra$-$ obtained by identifying two events in a probability space $(\Omega,\F,P)$ if, and only if, the symmetric difference of which is $P$-negligible.

For the convenience of the reader, let us fix some notation. Given a family $\{a_i\}_{i\in I}$ in $\mathcal{A}$, its supremum is denoted by $\vee_{i\in I} a_i$ and its infimum by $\wedge_{i\in I} a_i$. For $a\in\A$ we define the set of partitions of $a$, which is given by
\[
p(a):=\{ \{a_i\}_{i\in I} \subset \A \:;\: a=\vee a_i\textnormal{, }a_i \wedge a_j = 0, \textnormal{ for all }i\neq j\textnormal{, }i,j\in I  \}. 
\]
Note that we allow $a_i=0$ for some $i\in I$. We will also denote by $\A_a:=\{b\in\A\:;\: b\leq a \}$ the trace of $\A$ on $a$, which is a complete Boolean algebra as well. 

Let us recall the notion of conditional set:

\begin{defn}
\cite[Definition 2.1]{key-7}
\label{def: condSet}
Let $E$ be a non empty set and let $\A$ be a complete Bolean algebra. A \textit{conditional set} of $E$ and $\A$ is a collection $\textbf{E}$ of objects $x|a$ for $x\in E$ and $a\in\A$ satisfying the following three axioms: 

\begin{enumerate}
	\item if $x,y\in E$ and $a,b\in\A$ with $x|a=y|b$, then $a=b$;
	\item (Consistency) if $x,y\in E$ and $a,b\in\A$ with $a\leq b$, then $x|b=y|b$ implies $x|a=y|a$;
	\item (Stability) if $\{a_i\}_{i\in I}\in p(1)$ and $\{x_i\}_{i\in I}\subset E$, then there exists an unique $x\in E$ such that $x|a_i=x_i|a_i$ for all $i\in I$. 
\end{enumerate}
    
The unique element $x\in E$ provided by 3, is called the \textit{concatenation} of the family $\{x_i\}$ along the partition $\{a_i\}$, and is denoted by $\sum_{i\in I} x_i|a_i$, or $\sum x_i|a_i$.
\end{defn}  

In which follows, we shall recall the basic notions of conditional set theory and the notation that will be employed throughout this paper, which is essentially  the same as \cite{key-7} with a few exceptions that will be explained. Since the theory of conditional sets is an extensive theoretical development, there is no room to give a detailed review of all the notions. For this reason, in many cases we will just mention some concepts with an exact reference to the original definition in \cite{key-7}. 

Let us start by recalling an important example of conditional set, which is referred to as the \textit{conditional set of step functions} of a non empty set $X$. Namely, for given a non empty set $X$, the conditional set $\textbf{E}$ of step functions of $X$ is a collection of objects $x|a$ where $a\in\A$ and $x$ is a formal sum $x=\sum x_i|a_i$ with $\{a_i\}\in p(1)$ and $x_i\in X$ for all $i$. We do not give  the formal construction; instead, we refer the reader to 5 of \cite[Examples 2.3]{key-7}. 

As particular cases, when either $X=\N$ or $X=\Q$, we obtain the \textit{conditional natural numbers} $\textbf{N}$, or the \textit{conditional rational numbers} $\textbf{Q}$, respectively. Likewise, we have the generating sets, which are given by:
\[
\begin{array}{ccc}
N=\left\{ \sum n_i|a_i  \:;\: n_i\in\N,\:\{a_i\}\in p(1) \right\}  & \textnormal{ and } & Q=\left\{ \sum q_i|a_i  \:;\: q_i\in\Q,\:\{a_i\}\in p(1) \right\}.
\end{array}
\]

\begin{defn}
\cite[Definition 2.5]{key-7}
Let $\textbf{E}$ be a conditional set of a non empty set $E$. A non empty subset $F$ of $E$ is called \textit{stable} if
\[
F=\left\{ \sum x_i|a_i \:;\: \{a_i\}\in p(1),\: x_i\in F\textnormal{ for all }i \right\}.
\]
$S(\textbf{E})$ stands for the set of all stable subsets $F$  of $E$.
\end{defn}

The \textit{stable hull} of a non empty subset $F$ of $E$ is introduced in \cite{key-7} as 
\[
s(F):=\left\{\sum x_i|a_i \:;\: \{a_i\}\in p(1),\: x_i\in F\textnormal{ for all }i\right\},
\]
which is the smallest stable subset containing $F$.

It is known from \cite{key-7} that every set $F\in S(\textbf{E})$ generates a conditional set 
\[
\textbf{F}:=\left\{x|a\:;\: x\in F,\:a\in\A\right\}.
\]

For given a conditional set $\textbf{E}$, $P(\textbf{E})$ denotes the collection of all conditional sets $\textbf{F}$ generated by $F\in S(\textbf{E})$, and the \textit{conditional power set} is defined by   
\[
\textbf{P}(\textbf{E}):=\left\{\textbf{F}|a=\{x|b\:;\: x\in F,\:b\leq a\}\:;\:\textbf{F}\in P(\textbf{E}),\:a\in\A\right\},
\] 
which is a conditional set of $P(\textbf{E})$ (see \cite[Definition 2.7]{key-7}).

Drapeau et al. \cite{key-7} also observed that every element $\textbf{F}|a$ is a conditional set of $F|a:=\{x|a\:;\: x\in F\}$ and $\A_a$,  with the conditioning $(x|a)|b:=x|b$ for $b\leq a$. Such conditional sets are called \textit{conditional subsets} of $\textbf{E}$. In particular, we have $\textbf{E}|0$, which will be referred to as \textit{null conditional set}.

\begin{defn}
\cite[Definition 2.8]{key-7}
Suppose that $\textbf{E}$ and $\textbf{F}$ are conditional sets of $E,\A$ and $F,\mathcal{B}$, respectively. Then, $\textbf{F}$ is said to be conditionally included, or conditionally contained, in $\textbf{E}$ if $\mathcal{B}=\A_a$ for some $a\in\A$ and $\textbf{F}=\textbf{H}|a$ for some $\textbf{H}\in P(\textbf{E})$. In that case, we use the notation $\textbf{F}\sqsubset\textbf{E}$, and we say that $\textbf{F}$ is a conditional subset of $\textbf{E}$ on $a$.
\end{defn}

As in \cite{key-7}, the notation $\textbf{F}$ or $\textbf{F}|a$ will be chosen depending on the context.

Drapeau et al. \cite{key-7}  pointed out that $\sqsubset$ is a partial order on $\textbf{P}(\textbf{E})$ with greatest element $\textbf{E}=\textbf{E}|1$ and least element $\textbf{E}|0$. In fact, we see that $\sqsubset$ is  the classical set inclusion for elements of $\textbf{P}(\textbf{E})$. 

For any non empty subset $F$ of $E$, we will denote by $\textbf{s}(F)$ the conditional subset generated by the stable hull of $F$. 

Given $x\in E$, we will employ the notation $\textbf{x}$ for the object $x|1$. Such an element will be called \textit{conditional element} of $\textbf{E}$.\footnote[1]{Drapeau et al. \cite{key-7} introduced the notion of conditional element in a different way. Namely, every $x\in E$ defines a conditional subset $\{x|a\:;\: a\in\A\}$. Due to consistency, there is a bijection between the collection of these conditional subsets and the collection of objects $x|1$ with $x\in E$, therefore we can rewrite a new definition of conditional element $\textbf{x}$ as the object $x|1$. By doing so, we obtain that a conditional element is in fact an element of $\textbf{E}$, and we can use the convenient notation $\textbf{x}\in\textbf{E}$, which is not valid for the notion introduced in \cite{key-7}.} Since $\textbf{F}|a$ is a conditional set of $F|a$, a conditional element of $\textbf{F}|a$ is an elements of the form $x|a$ with $x\in F$, which is denoted by $\textbf{x}|a$. In general, if we choose the notation $\textbf{F}$ for a conditional subset, we will use the notation $\textbf{x}\in\textbf{F}$ for its conditional elements.

From time to time, some conditional subsets will be required to be defined by describing their conditional elements. For instance, suppose that $\phi$  is a certain statement which can be true of false for the conditional elements of $\textbf{E}$. Also, assume that the statement $\phi$ is true for at least a conditional element $\textbf{x}_0\in \textbf{E}$. Since the family $\{\textbf{x} \in \textbf{E} \:;\: \phi(\textbf{x})\textnormal{ is true}\}$ is not generally a conditional set, we will employ the following formal set-builder notation for conditional subsets:
\[
                    [\textbf{x} \in \textbf{E} \:;\: \phi(\textbf{x})\textnormal{ is true}] := \textbf{s} \left(\left\{x \:;\: \phi(\textbf{x})\textnormal{ is true} \right\}\right) .
\]

Note that the stable hull of a subset $L$ of $E$ can be defined only if $L$ is non empty. For this reason we supposed $\phi$ to be true for some conditional element of $\textbf{E}$. 

We also have operations for conditional subsets of a conditional set $\textbf{E}$. Namely, we have the \textit{conditional union}, the \textit{conditional intersection}, and the \textit{conditional complement}, which are denoted by $\sqcup$, $\sqcap$ and ${}^\sqsubset$, respectively. We do not include the construction; instead, we refer to the proof of \cite[Theorem 2.9]{key-7}. Moreover, $(\textbf{P}(\textbf{E}), \sqcup,\sqcap,{}^\sqsubset, \textbf{E}, \textbf{E}|0)$ is a complete Boolean algebra (see \cite[Corollary 2.10]{key-7}).

\begin{defn}
\cite[Definition 2.17]{key-7} Let $\textbf{E},\textbf{F}$ be conditional sets. A function $f:E\rightarrow F$ is said to be \textit{stable} if 
\[
\begin{array}{cc}
f\left(\sum x_i|a_i \right)=\sum f(x_i)|a_i, & \textnormal{ for all }\{a_i\}\in p(1)\textnormal{ and every family }\{x_i\}\subset E. 
\end{array}
\]
A conditional subset $\textbf{G}_\textbf{f}$ of $\textbf{E}\Join\textbf{F}$ on $1$ is the \textit{graph of a conditional function} $\textbf{f}:\textbf{E}\rightarrow\textbf{F}$ if $G_f$ is the graph of a stable function $f:E\rightarrow F$.

For $\textbf{x}\in\textbf{E}$, we define $\textbf{f}(\textbf{x}):=f(x)|1$.

For $\textbf{A}|a\sqsubset \textbf{E}$ its \textit{conditional image} is defined by $\textbf{f}(\textbf{A}|a):=\left[ \textbf{f}(\textbf{x}) \:;\: \textbf{x}\in \textbf{A} \right]|a.$

For $\textbf{B}|b\sqsubset\textbf{F}$ the \textit{conditional preimage} is $\textbf{f}^{-1}(\textbf{B}|b)=\textbf{U}|c$, where
\[
c:=\vee\left\{a\:;\: a\leq b \:;\: \exists x\in E,\: y\in U\textnormal{ such that }f(x)|a=y|a\right\},
\]
\[
U:=\left\{ x\in E \:;\: \exists y\in V,\:f(x)|c=y|c\right\}.
\]
A conditional function $\textbf{f}:\textbf{E}\rightarrow\textbf{F}$ is conditionally injective if for every pair $x,x'\in E$ with $x|a\neq x'|a$ for all $a\neq 0$ it holds that $f(x)|a\neq f(x')|a$ for all $a\neq 0$; it is conditionally surjective if $f$ is surjective; and it is conditionally bijective if it is conditionally injective and conditionally surjective.
\end{defn}

The notions of \textit{conditionally countable} and \textit{conditionally finite} conditional subset are introduced in \cite[Definition 2.23]{key-7}.

Other relevant notions are the concepts of \textit{stable family} and \textit{conditional family} (we refer to \cite[Definition 2.20]{key-7}).

The \textit{conditional Cartesian product} is introduced in \cite[Definition 2.14]{key-7}.  Let be given a non empty family $\{\textbf{E}_i\}_{i\in I}$ of conditional sets of $E_i$ and $\A$, then we will denote their \textit{conditional product} as follows:
\[
\conprod_{i\in I} \textbf{E}_i=\left\{ (x_i|a)_{i\in I} \:;\: \: x_i|a\in\textbf{E},\: a\in\A  \right\}.
\]

For a finite family of conditional sets, the conditional product will be denoted by $\textbf{E}_1\Join...\Join\textbf{E}_n$.

For given two conditional sets $\textbf{E},\textbf{I}$, Drapeau et al.\cite{key-7} defined
\[
\textbf{E}^\textbf{I}:=\left[\{\textbf{x}_\textbf{i}\}_{\textbf{i}\in\textbf{I}}\:;\: \{\textbf{x}_\textbf{i}\}_{\textbf{i}\in\textbf{I}}\textnormal{ is a conditional family}  \right].
\]

Let us introduce the following version of the conditional Cartesian product for conditional families of conditional sets, which is based on the conditional axiom of choice (see \cite[Theorem 2.26]{key-7}) and will be used in some point of this paper. Namely, let $\textbf{E}$ be a conditional set and $\{\textbf{E}_\textbf{i}\}_{\textbf{i}\in\textbf{I}}$ a conditional family of conditional subsets of $\textbf{E}$, we define
\[
\conprod_{\textbf{i}\in \textbf{I}} \textbf{E}_\textbf{i}:=
\left[\{\textbf{x}_\textbf{i}\}\in (\underset{\textbf{i}\in\textbf{I}}\sqcup\textbf{E}_\textbf{i})^\textbf{I}\:;\:\textbf{x}_\textbf{i}\in \textbf{E}_\textbf{i}\textnormal{ for each }\textbf{i}\in \textbf{I}\right].
\]

We also have the notions of \textit{conditional binary relation}, \textit{conditional direction}, \textit{conditional partial order} and \textit{conditional total order} (see \cite[Definition 2.15]{key-7}). Further, following \cite{key-7}, for a conditional partially ordered set $(\textbf{E},\leq)$, \textit{conditional upper} (resp. \textit{lower}) \textit{bounds} and \textit{conditional supremums} (resp. \textit{infimums}) are defined as the corresponding notions for classical partially ordered sets. As particular case, we have that a conditional total order $\leq$ can be defined on the conditional rational numbers $\textbf{Q}$ (see 1 of \cite[Examples 2.16]{key-7}).

\begin{defn}
\cite[Example 2.21]{key-7}
A conditional family $\{\textbf{x}_\textbf{i}\}_{\textbf{i}\in\textbf{I}}$ of conditional elements of $\textbf{E}$, with $(\textbf{I},\leq)$ a conditional direction, is called a \textit{conditional net}. If $\textbf{I}=\textbf{N}$, a conditional family $\{\textbf{x}_\textbf{n}\}$ is called a \textit{conditional sequence}. 

Suppose that $\{\textbf{x}_\textbf{n}\}$ and $\{\textbf{y}_\textbf{n}\}$ are conditional sequences. Then, $\{\textbf{y}_\textbf{n}\}$ is said to be a \textit{conditional subsequence} of $\{\textbf{x}_\textbf{n}\}$ if there exists a conditional sequence $\{\textbf{n}_\textbf{k}\}$ in $\textbf{N}$, for which $\textbf{k}<\textbf{k}'$ implies that $\textbf{n}_\textbf{k}<\textbf{n}_\textbf{k}'$, such that $\textbf{y}_\textbf{k}=\textbf{x}_{\textbf{n}_\textbf{k}}$ for all $\textbf{k}\in\textbf{N}$.  
\end{defn}

Let us list more relevant notions. For saving space, we will just give a reference of \cite{key-7}:  

We have \textit{conditional topologies} (and \textit{conditional topological spaces}), \textit{conditionally open sets}, \textit{conditionally closed sets} and \textit{conditional topological bases} (see \cite[Definition 3.1]{key-7}); the notions of \textit{conditional interior} and \textit{conditional closure} (see \cite[Definition 3.3]{key-7}); \textit{conditional neighborhoods} and \textit{conditional bases of neighborhoods} (see \cite[Definition 3.4]{key-7}); \textit{conditionally continuous functions} between conditional topological spaces (see \cite[Definition 3.8]{key-7}). 

We will also say that a conditional function $\textbf{f}:\textbf{E}\rightarrow\textbf{F}$ between conditional topological spaces is \textit{conditionally open}, if $\textbf{f}(\textbf{O})$ is conditionally open in $\textbf{F}$ whenever $\textbf{O}$ is conditonally open in $\textbf{E}$. Further, $\textbf{f}$ is a \textit{conditional homeomorphism} if it is conditionally bijective, conditionally continuous and conditionally open.     

Let $(\textbf{E},\mathcal{T})$ be a conditional topological space, and let $\textbf{F}$ be a conditional subset of $\textbf{E}$, then  $\overline{\textbf{F}}^{\mathcal{T}}$, or simply $\overline{\textbf{F}}$, stands for the conditional closure of $\textbf{F}$.

In \cite{key-7} the following notions are also introduced:

\begin{defn}
Let $(\textbf{E},\mathcal{T})$ be a conditional topological space. A conditional subset $\textbf{F}\sqsubset\textbf{E}$ is \textit{conditionally dense} if $\overline{\textbf{F}}=\textbf{E}$. If $\textbf{E}$ has a conditionally countable subset $\textbf{F}$ which is also conditionally dense, then $\textbf{E}$ is called \textit{conditionally separable}. 
\end{defn}

Also, a notion which will be key in this paper is the following:

\begin{defn}
\cite[Definition 3.24]{key-7} 
A conditionally topological space $(\textbf{E},\mathcal{T})$ is \textit{conditionally compact} if for every conditional family $\{\textbf{O}_\textbf{i}\}_{\textbf{i}\in\textbf{I}}$ of conditionally open conditional subsets such that $\textbf{E}=\sqcup \textbf{O}_\textbf{i}$, there exists a conditionally finite subset $\textbf{F}$ of $\textbf{I}$ such that $\textbf{E}=\underset{\textbf{i}\in\textbf{F}}\sqcup \textbf{O}_\textbf{i}$.
\end{defn}

It is also important to recall the following notions:

\begin{defn}
\cite[Definition 3.19]{key-7}
Let $(\textbf{E},\mathcal{T})$ be a conditional topological space, $\{\textbf{x}_\textbf{i}\}$ a conditional net and $\textbf{x}\in\textbf{E}$. Then
\begin{enumerate}
	\item $\textbf{x}$ is a \textit{conditional limit point} of $\{\textbf{x}_\textbf{i}\}$ if for every conditional neighborhood $\textbf{U}$ of $\textbf{x}$ there exists $\textbf{j}$ such that $\textbf{x}_\textbf{i}\in\textbf{U}$ for all $\textbf{i}\geq\textbf{j}$;
	\item $\textbf{x}$ is a conditional cluster point of $\{\textbf{x}_\textbf{i}\}$ if for every conditional neighborhood $\textbf{U}$ of $\textbf{x}$ and every $\textbf{i}$ there exists $\textbf{j}\geq\textbf{i}$ such that $\textbf{x}_\textbf{j}\in\textbf{U}$.
\end{enumerate}
\end{defn}  

If $\textbf{x}$ is the unique conditional limit point of $\{\textbf{x}_\textbf{i}\}$, we will say that $\{\textbf{x}_\textbf{i}\}$ is \textit{conditionally convergent}, and it \textit{conditionally converges} to $\textbf{x}$. We will write $\nlim \textbf{x}_\textbf{i}=\textbf{x}$.

\begin{defn}
A conditionally topological space $(\textbf{E},\mathcal{T})$ is \textit{conditionally sequentially compact} if every conditional sequence in $\textbf{E}$ has a conditional subsequence with a \textit{conditional limit point}.\footnote[2]{In \cite{key-7}, in the context of conditional metric spaces, it is introduced the notion of \textit{conditional sequentially compactness} in terms  of conditional cluster points instead of limit points. However, in the context of conditional metric spaces, both definitions turn out to equivalent. In fact, in \cite{key-7} this equivalence is implicitly assumed. For instance, in (iii) implies (iv) of Theorem 4.6 of \cite{key-7}, it is supposed that a conditional sequence has a conditionally convergent subsequence, provided that the  conditional space is conditionally sequentially compact.} 
\end{defn}

In \cite{key-7} the \textit{conditional real numbers} are introduced. There is provided a construction of a \textit{conditionally totally ordered field} $(\textbf{R},+,\cdot,\leq)$ which is \textit{conditionally Dedekind complete}; $\textbf{Q}$ is conditionally included in $\textbf{R}$; and the conditional sum, conditional product and conditional order of $\textbf{R}$ extend the conditional sum, conditional product and conditional order of $\textbf{Q}$, respectively (see Definitions 4.1 and 4.3 of \cite{key-7} for terminology, and 2 of Examples 4.2 for reviewing the construction).

We will use the following notion:
\[
\begin{array}{cc}
\textbf{R}^+:=\left[\textbf{r}\in\textbf{R}\:;\: \textbf{r}\geq\textbf{0}\right],  &  \textbf{R}^{++}:=\left[ \textbf{r}\in\textbf{R} \:;\: \textbf{r}>\textbf{0} \right].
\end{array}
\]
Also, for $\textbf{r}\in\textbf{R}$, let us put $a:=\vee\left\{b\in\A\:;\: r|b\in [\textbf{0}]^\sqsubset \right\}$ and $s:=\frac{1}{s}|a + 0|a^c$. We define the \textit{conditional inverse} of $\textbf{r}$ as $\textbf{r}^{-1}:=s|1$.

Drapeau et al.\cite{key-7} provided the definitions of \textit{conditional metric} and \textit{conditional metric space} (see \cite[Definition 4.5]{key-7}). Also, it is introduced the notion of conditionally topology \textit{induced} by a conditional metric, which is \textit{conditionally Hausdorff}.

For given a conditional metric space $(\textbf{E},\textbf{d})$, a conditional sequence $\{\textbf{x}_\textbf{n}\}$ in $\textbf{E}$ is \textit{conditionally Cauchy}, if for each $\textbf{r}\in\textbf{R}^{++}$ there exists $\textbf{n}\in\textbf{N}$ such that $\textbf{d}(\textbf{x}_\textbf{p},\textbf{x}_\textbf{q})\leq\textbf{r}$ for all $\textbf{p},\textbf{q}>\textbf{n}$. $(\textbf{E},\textbf{d})$ is said to be \textit{conditionally complete} if every conditionally Cauchy sequence has a conditional limit.

We will focus our study on \textit{conditional vector spaces} (see \cite[Definion 5.1]{key-7}). 

For given a conditional vector space $\textbf{E}$, we will denote by 
\[
\supp(\textbf{E}):=\vee\{a\in\A\:;\: \exists x|a\in\textbf{E},\:  x|a\in [\textbf{0}]^\sqsubset\}
\]
 the \textit{support} of $\textbf{E}$. For the sake of simplicity and without loss of generality, we will assume hereafter that the support of every conditional vector space considered is $1$.    

Likewise, for $\textbf{x}\in\textbf{E}$ we define its support, which is denoted by $\supp(\textbf{x}):=\vee\{a\in\A\:;\: x|a\in [\textbf{0}]^\sqsubset\}$.

In \cite{key-7} we can find the definition of \textit{conditional subspace} (see \cite[Definion 5.1]{key-7}). Also, there is introduced the following notation: let $\textbf{E}$ be a conditional vector space, and let there be given $\textbf{n}\in\textbf{N}$ with $n=\sum n_i|a_i$, $\{a_i\}\in p(1)$, and $n_i\in\N$ for all $i$; and conditionally finite families $\{\textbf{r}_\textbf{k}\}_{\textbf{1}\leq \textbf{k}\leq\textbf{n}}\sqsubset\textbf{R}$ and $\{\textbf{x}_\textbf{k}\}_{\textbf{1}\leq \textbf{k}\leq\textbf{n}}\sqsubset\textbf{E}$. We define
\[
\begin{array}{cc}
\sum_{\textbf{1}\leq\textbf{k}\leq\textbf{n}}\textbf{n}_\textbf{k}\textbf{x}_\textbf{k}=\textbf{r}, & \textnormal{ where }r:=\sum_{i\in I}\left(\sum_{k=1}^{n_i} r_k x_k\right)|a_i. 
\end{array}
\] 
Also, for given a conditional subset $\textbf{F}$ of $\textbf{E}$, in \cite{key-36} it is defined
\[
\spa_\textbf{R}\:\textbf{F}:=\left[\sum_{\textbf{1}\leq\textbf{k}\leq\textbf{n}}\textbf{n}_\textbf{k}\textbf{x}_\textbf{k} \:;\: \{\textbf{x}_\textbf{k}\}_{\textbf{1}\leq\textbf{k}\leq\textbf{n}}\sqsubset\textbf{E},\: \{\textbf{r}_\textbf{k}\}_{\textbf{1}\leq\textbf{k}\leq\textbf{n}}\sqsubset\textbf{R},\: \textbf{n}\in\textbf{N} \right]. 
\]

\begin{defn}
\cite[Definion 5.1]{key-7}
Suppose that $\textbf{F}$ is a conditional subset of a conditional vector space $\textbf{E}$. Then $\textbf{F}$ is said to be:
\begin{enumerate}
\item conditionally convex if $\textbf{r} \textbf{x}_1 + (1 - \textbf{r})\textbf{x}_2\in \textbf{U}$ for $\textbf{x}_1,\textbf{x}_2\in \textbf{U}$ and $\textbf{r}\in\textbf{R}$, with $\textbf{0}\leq \textbf{r}\leq \textbf{1}$;
\item conditionally absorbent if for every $\textbf{x}\in\textbf{E}$, there is a $\textbf{r}\in\textbf{R}^{++}$ such that $\textbf{x}\in \textbf{r} \textbf{U}$;
\item conditionally balanced if $\textbf{r}\textbf{x}\in\textbf{U}$ for $\textbf{x}\in \textbf{U}$ and $\textbf{r}\in\textbf{R}$ with $|\textbf{r}|\leq \textbf{1}$.
\end{enumerate}
\end{defn}  

\begin{defn}
\cite[Definition 5.4]{key-7}
\label{def: condLocSpac}
Let $\textbf{E}[\mathcal{T}]$ be a conditional topological space. The conditional topology $\mathcal{T}$ is said to be conditionally locally convex  if there is a conditional neighborhood base $\mathcal{U}$ of $\textbf{0}\in\textbf{E}$ such that each $\textbf{U}\in\mathcal{U}$ on $1$ is  conditionally convex. In this case, $\textbf{E}[\mathcal{T}]$ is called a conditional locally convex space.
\end{defn}

An example of a conditional locally convex space is a \textit{conditionally normed space} (see \cite[Definition 5.11]{key-7}). If $(\textbf{E},\Vert\cdot\Vert)$ is a conditionally normed space, then the conditional norm defines a conditional metric $\textbf{d}(\textbf{x},\textbf{y}):=\Vert \textbf{x}-\textbf{y}\Vert$, which induces a conditional topology. A conditionally normed space which is conditionally complete is called a \textit{conditional  Banach space}. For instance, $(\textbf{R},|\cdot|)$ is a conditional Banach space. 

In some points of this paper, we will use \textit{conditionally infinite sums} or \textit{conditional series}. Namely, if $\{\textbf{x}_\textbf{n}\}$
 is a conditional sequence in a conditionally normed space $(\textbf{E},\Vert\cdot\Vert)$, we define the conditional sequence $\{\textbf{s}_\textbf{n}\}$ which is given by $\textbf{s}_\textbf{n}:=\sum_{\textbf{1}\leq\textbf{k}\leq\textbf{n}} \textbf{x}_\textbf{k}$. If $\{\textbf{s}_\textbf{n}\}$ conditionally converges, we will write
\[
\begin{array}{cc}
\sum_{\textbf{k}\geq\textbf{1}} \textbf{x}_\textbf{k}:=\nlim \textbf{s}_\textbf{n},\textnormal{ provided the conditional limit exists}.
\end{array}
\]

For given a conditionally normed space $(\textbf{E},\Vert\cdot\Vert)$, we will denote by
\[
\textbf{B}_\textbf{r}(\textbf{x}):=\left[\textbf{y}\in\textbf{E}\:;\:\Vert\textbf{x}-\textbf{y}\Vert\leq \textbf{r}\right]
\] 
the \textit{conditional ball} of radius $\textbf{r}\in\textbf{R}^{++}$ centered at $\textbf{x}\in\textbf{E}$.
 We also denote by $\textbf{B}_\textbf{E}:=\left[\textbf{x}\in\textbf{E} \:;\: \Vert\textbf{x}\Vert\leq \textbf{1}\right]$ the \textit{conditional unit ball} of $\textbf{E}$. Similarly, $\textbf{S}_\textbf{E}$ stands for the \textit{conditional unit sphere}, i.e. $\textbf{S}_\textbf{E}:=\left[\textbf{x}\in\textbf{E} \:;\: \Vert\textbf{x}\Vert=\textbf{1}\right]$.

Drapeau et al.\cite{key-7} also introduced the notion of \textit{conditionally linear function} between conditional vector spaces (see \cite[Definition 5.1]{key-7}).

Let $\textbf{E},\textbf{F}$ be conditional normed spaces. We define 
\[
\textbf{L}(\textbf{E},\textbf{F}):=\left[\textbf{T}:\textbf{E}\rightarrow\textbf{F}\:;\: \textbf{T}\textnormal{ conditionally linear continuous} \right].
\]
Then, it is not surprising and proved in a similar manner to the classical case that 
\[
\Vert \textbf{T}\Vert:=\nsup\left[ \Vert\textbf{T}(\textbf{x})\Vert \:;\: \textbf{x}\in \textbf{B}_{\textbf{E}}\right]=\nsup\left[\Vert\textbf{T}(\textbf{x})\Vert\:;\: \textbf{x}\in \textbf{S}_{\textbf{E}}\right].
\]
 defines a conditional norm on $\textbf{L}(\textbf{E},\textbf{F})$. Moreover, inspection shows that if $\textbf{F}$ is conditionally Banach, then $\textbf{L}(\textbf{E},\textbf{F})$ is conditionally Banach. 

The particular case of $\textbf{F}=\textbf{R}$ is introduced in \cite{key-7}. We can consider the \textit{conditional topological dual space} defined by $\textbf{E}^*:=\textbf{L}(\textbf{E},\textbf{R})$, which is conditionally Banach as $\textbf{R}$ is conditionally complete. 

The following notion will be needed later:
\begin{defn} 
A conditional subset $\textbf{D}$ of $\textbf{E}^*$ on $1$, is said to be \textit{conditionally total} if for each $\textbf{x}\in\textbf{E}$ such that $\textbf{x}^*(\textbf{x})=\textbf{0}$ for all $\textbf{x}^*\in\textbf{D}$ it holds that $\textbf{x}=\textbf{0}$. 
\end{defn}

We also recall the conditional versions of Hahn-Banach type extension and separation theorems proved in \cite{key-7}:

\begin{thm}
\label{thm: HB}
\cite[Theorem 5.3]{key-7}
Let $\textbf{E}$ be a conditionally vector space, $\textbf{p} : \textbf{E}\rightarrow\textbf{R}$ a conditionally convex function, $\textbf{F}$ a
conditional subspace of $\textbf{E}$, and $\textbf{x}^*_0\in\textbf{F}^*$ such that $\textbf{x}^*_0(\textbf{y}) \leq \textbf{p}(\textbf{y})$ for every $\textbf{y}$ in $\textbf{Y}$. Then there exists $\textbf{x}^*\in\textbf{E}^*$ such that $\textbf{x}^*(\textbf{x})\leq \textbf{p}(\textbf{x})$ for every $\textbf{x}$ in $\textbf{E}$ and $\textbf{x}^* (\textbf{y}) = \textbf{x}^*_0 (\textbf{y})$ for all $\textbf{y}$ in $\textbf{F}$.
\end{thm}

\begin{thm}
\label{thm: HB2}
\cite[Theorem 5.5]{key-7}
Let $\textbf{E}$ be a conditional locally convex space, and let $\textbf{C},\textbf{D}$ be
two conditional convex sets of $\textbf{E}$ on $1$, such that $\textbf{C}\sqcap\textbf{D}=\textbf{E}|0$. Then:
\begin{enumerate}
	\item If $\textbf{C}$ is conditionally open, then there exists $\textbf{x}^*\in\textbf{E}^*$ such that $\textbf{x}^*(\textbf{x}) < \textbf{x}^* (\textbf{y})$ for every $(\textbf{x},\textbf{y})\in\textbf{C}\Join\textbf{D}$.
	\item If $\textbf{C}$ is conditionally compact and $\textbf{D}$ conditionally closed, then there exist $\textbf{x}^*\in\textbf{E}^*$ and $\textbf{r}\in\textbf{R}^{++}$ such that $\textbf{x}^*(\textbf{x}) + \textbf{r} < \textbf{x}^* (\textbf{y})$ for every $(\textbf{x},\textbf{y})\in\textbf{C}\Join\textbf{D}$.
\end{enumerate}
\end{thm}

Finally, let us give some comments on conditionally finitely generated vector spaces which will be needed later:

\begin{defn}
A conditional vectorial space $\textbf{E}$ is finitely generated is there exists a conditionally finitely subset $[\textbf{x}_\textbf{k}\:;\:\textbf{1}\leq\textbf{k}\leq\textbf{n}]$ of $\textbf{E}$ such that $\textbf{E}=\spa_\textbf{R}[\textbf{x}_\textbf{k}\:;\:\textbf{1}\leq\textbf{k}\leq\textbf{n}]$.  
\end{defn}

A conditional version of Heine-Borel theorem (see \cite[Theorem 5.5.8]{key-36}) was proved for the conditional normed space $\textbf{R}^\textbf{n}:=\textbf{R}^{[\textbf{k}\in\textbf{N}\:;\: \textbf{1}\leq\textbf{k}\leq\textbf{n}]}.$

 The Proposition below shows that this result is still valid for conditionally finitely generated  normed spaces. This is not surprising, hence the proof has been placed in an appendix at the end of this paper (see Proposition \ref{prop: Heine-BorelExt}):

\begin{prop}
\label{prop: Heine-Borel}
Let $(\textbf{E},\Vert\cdot\Vert)$ be conditionally normed space which is also conditionally finitely generated. Then, every conditionally bounded and conditionally closed subset $\textbf{K}$ of $\textbf{E}$ is conditionally compact.
\end{prop}

\section{Conditional versions of some basic results of functional analysis}

In this section we will show conditional versions of well-known results of classical functional analysis,  which will be needed to prove the main results. 

\subsection{Conditional versions of Baire Category theorem and Uniform  Boundedness Principle}

\begin{thm}
\label{thm: baire}[Conditional version of Baire Category theorem]
Suppose that a conditionally complete metric space $(\textbf{E},\textbf{d})$ is the conditional union of a conditionally countable family $\{\textbf{E}_\textbf{n}\}_{\textbf{n}\in\textbf{N}}$ of conditionally closed subsets, then there exist a conditional ball $\textbf{B}_{\textbf{r}}(\textbf{x})$ on $1$ and $\textbf{n}\in\textbf{N}$ such that $\textbf{B}_{\textbf{r}}(\textbf{x})\sqsubset \textbf{E}_\textbf{n}$.  
\end{thm}
\begin{proof}

Suppose $\textbf{E}=\sqcup_{\textbf{n}\in\textbf{N}} \textbf{E}_{\textbf{n}}$ and define 
\[
d:=\vee\left\{ a\in\A \:;\: \textnormal{ there exists } \textbf{B}_{\textbf{r}}(\textbf{x})|a\sqsubset E_{\textbf{n}}
\textnormal{ for some }\textbf{n}\in\textbf{N}\textnormal{ and }\textbf{r}\in\textbf{R}^{++} \right\}.
\]
We claim that $d$ is attained. Indeed, let us choose some family $\{a_j\}\subset\A$ with $d=\vee a_j$ and such that for each $j$ there exist $\textbf{r}_j\in\textbf{R}^{++}$, $\textbf{n}_j\in\textbf{N}$, and $\textbf{x}_j\in\textbf{E}$ so that $\textbf{B}_{\textbf{r}_j}(\textbf{x}_j)|a\sqsubset \textbf{E}_{\textbf{n}_j}$ for all $j$. Due to the well-ordering theorem, we can find $\{d_j\}\in p(d)$ such that $d_j\leq a_j$. We have that $\textbf{B}_{\textbf{r}_j}(\textbf{x}_j)|d_j\sqsubset \textbf{E}_{\textbf{n}_j}$ for each $j$. If we take $r=\sum r_j|d_j\in N$, $x=\sum x_j|d_j\in E$ and $n=\sum n_j|d_j\in N$, it follows that $\textbf{B}_{\textbf{r}}(\textbf{x})|d\sqsubset \textbf{E}_{\textbf{n}}$.

If $d=1$ we are done as $d$ is attained. In other case, by consistency, we may assume w.l.g. $d=0$ by arguing on $d^c$.
 
If so, let us put $\textbf{E}_1:=\textbf{E}_{\textbf{1}}$. We claim that the conditional set $\textbf{E}_1^{\sqsubset}$ is on $1$. Indeed, let
\[
d_1:=\vee\left\{a\in\A\:;\: \textnormal{ there exists }x\in E\textnormal{ with }x|a\in \textbf{E}_1^{\sqsubset}\right\}.
\]

Then, since $\textbf{E}|d_1^c\sqsubset\textbf{E}_1\sqcup\textbf{E}_1^{\sqsubset}$ and $\textbf{E}|d_1^c\sqcap\textbf{E}_1^{\sqsubset}=\textbf{E}|0$, it follows $\textbf{E}|d_1^c\sqsubset \textbf{E}_1$, which implies $d_1^c=0$, because $d=0$. Given that $\textbf{E}_1^{\sqsubset}$ is conditionally open, it must therefore conditionally contain a conditional ball $\textbf{B}_1:=\textbf{B}_{\textbf{r}_1}(\textbf{x}_1)$ with $\textbf{0}<\textbf{r}_1<\textbf{1}/\textbf{2}$, which can be taken on $1$. Now look at $\textbf{E}_2:=\textbf{E}_{\textbf{2}}$. We claim that the conditional set $\textbf{E}_2^{\sqsubset}\sqcap \textbf{B}_{\textbf{r}_1/\textbf{2}}(\textbf{x}_1)$ is on $1$. Indeed, let us consider
\[
d_2:=\vee\left\{a\in\A\:;\: \textnormal{ there exists }x\in E\textnormal{ with }x|a\in \textbf{E}_2^{\sqsubset}\sqcap \textbf{B}_{\textbf{r}_1/\textbf{2}}(\textbf{x}_1)\right\}.
\]
By doing so, we obtain that $\textbf{B}_{\textbf{r}_1/\textbf{2}}(\textbf{x}_1)|d_2^c\sqsubset\textbf{E}_2$, and necessarily $d_2^c=0$, since $d=0$. Being $\textbf{E}_2^{\sqsubset}\sqcap\textbf{B}_{\textbf{r}_1/\textbf{2}}(\textbf{x}_1)$ conditionally open and on $1$, it conditionally contains a conditional ball $\textbf{B}_2:=\textbf{B}_{\textbf{r}_2}(\textbf{x}_2)$ with $\textbf{0}<\textbf{r}_2<\textbf{1}/\textbf{4}$. 

By induction, we obtain a sequence $\textbf{B}_k:=\textbf{B}_{\textbf{r}_k}(\textbf{x}_k)$ of conditional balls on $1$ such that $\textbf{0}<\textbf{r}_k<\textbf{1}/\textbf{2}^\textbf{k}$, $\textbf{B}_{k+1}\sqsubset \textbf{B}_{\textbf{r}_k /\textbf{2}}(\textbf{x}_k)$, and $\textbf{B}_k \sqcap \textbf{E}_{k}=\textbf{E}|0$ for all $k\in\N$.

For each $\textbf{n}=n|1\in\textbf{N}$ with $n=\sum_{i\in I} n_i|a_i$, where $\{a_i\}\in p(1)$ and $n_i\in\N$ for all $j$, we define $\textbf{x}_\textbf{n}:=x_n|1$ and $\textbf{z}_\textbf{n}:=z_n|1$ with $x_n:=\sum x_{n_i}|a_i$ and $r_n:=\sum {r}_{n_i}|a_i$. Then $\{\textbf{x}_\textbf{n}\}_{\textbf{n}\in\textbf{N}}$ and $\{\textbf{r}_\textbf{n}\}_{\textbf{n}\in\textbf{N}}$ are  conditional sequences.

For $(\textbf{n},\textbf{m})\in\textbf{N}\Join\textbf{N}$ with $n=\sum_i n_i|b_i$ and $m=\sum_j m_j|c_j$, let us define $d_{i,j}:=b_i\wedge c_j$ and $a:=\vee\{b\in\A\:;\: m|b=n|b\}$. Then 

\[
d(x_n,x_m)=d(x_{n\vee m},{x}_{{n}\wedge {m}})\leq 0|a + \sum_{i,j} \left(\sum_{k=n_i\wedge m_j}^{(n_i\vee m_j)-1}{d}({x}_{k},{x}_{k+1}) \right)|a^c\wedge d_{i,j}\leq \frac{1}{2^{m \wedge n}}.
\]   

Consequently, $\{\textbf{x}_\textbf{n}\}_{\textbf{n}\in\textbf{N}}$ is a conditional Cauchy sequence, and given that $\textbf{E}$ is conditionally complete, it conditionally converges to some $\textbf{x}\in \textbf{E}$.

Now, since for all $\textbf{n}<\textbf{m}$
\[
\textbf{d}(\textbf{x}_{\textbf{n}},\textbf{x})\leq \textbf{d}(\textbf{x}_{\textbf{n}},\textbf{x}_{\textbf{m}})+\textbf{d}(\textbf{x}_{\textbf{m}},\textbf{x})<\frac{\textbf{r}_\textbf{n}}{\textbf{2}}+\textbf{d}(\textbf{x}_{\textbf{m}},\textbf{x})
\]
and $\nlim_{\textbf{m}} \textbf{d}(\textbf{x}_{\textbf{m}},\textbf{x})=0$, it follows that $\textbf{d}(\textbf{x}_\textbf{n},\textbf{x})\leq\frac{\textbf{r}_\textbf{n}}{\textbf{2}}$. In particular, it means that $\textbf{x}\in \textbf{B}_{\textbf{r}_\textbf{n}}(\textbf{x}_{\textbf{n}})$ for every $\textbf{n}\in\textbf{N}$. 

On the other hand, since $\textbf{E}=\sqcup_{\textbf{n}\in\textbf{N}} \textbf{E}_\textbf{n}$, we have that $\textbf{x}\in\textbf{E}_\textbf{n}$ for some $\textbf{n}$. Let us take $b\in\A$, $b\neq 0$, such that $n|b=k|b$ for some $k\in\N$. Then $x|b\in \textbf{E}_k\sqcap\textbf{B}_k$. But this is a contradiction.

\end{proof}

\begin{thm}
\label{thm: uniformBoundeness}[Conditional version of the Uniform Boundedness Principle] Let $\textbf{E}$ be a conditional Banach space,  and let $\textbf{F}$ be a conditionally normed space. Suppose that $\textbf{S}$ is a conditional subset of $\textbf{L}\left(\textbf{E},\textbf{F}\right)$ on $1$, such that for every $\textbf{x}\in\textbf{E}$ there exists $\textbf{r}(\textbf{x})\in\textbf{R}$ with $|\textbf{T}(\textbf{x})|\leq \textbf{r}(\textbf{x})$ for all $\textbf{T}\in \textbf{S}$. Then there exists $\textbf{s}\in \textbf{R}^{++}$ such that $\Vert \textbf{T} \Vert\leq \textbf{s}$ for all $\textbf{T}\in \textbf{S}$. 
\end{thm}
\begin{proof}

For each $\textbf{n}\in\textbf{N}$ we define the following conditional set
\[
\textbf{E}_\textbf{n} := \left[ \textbf{x}\in \textbf{E} \:;\: \underset{\textbf{T}\in \textbf{S}}\nsup |\textbf{T}(\textbf{x})|\leq \textbf{n}\right].
\]

Since the function $n\mapsto E_n$ is stable, it can be defined a conditionally countable family $\{\textbf{E}_\textbf{n}\}_{\textbf{n}\in\textbf{N}}$. 

Then it is clear that $\textbf{E}=\sqcup_{\textbf{n}\in\textbf{N}} \textbf{E}_\textbf{n}$.

Let us show that $\textbf{E}_\textbf{n}$ is conditionally closed. For that, let us take a  conditional sequence $\{\textbf{x}_{\textbf{k}}\}_{\textbf{k}\in\textbf{N}}$ in $\textbf{E}_\textbf{n}$ which conditionally converges to $\textbf{x}$, and let us show that $\textbf{x}\in \textbf{E}_\textbf{n}$. Indeed, arguing by way of contradiction, assume that there exists $c\in\A$, $c\neq 0$, and $\textbf{T}\in \textbf{S}$ so that $|\textbf{T}(\textbf{x})||c> \textbf{n}|c$. Let us suppose $c=1$, by consistently arguing on $c$. Since $\{\textbf{T}(\textbf{x}_{\textbf{k}})\}$ conditionally converges to $|\textbf{T}(\textbf{x})|$, there exists some $\textbf{k}'\in\textbf{N}$ with $|\textbf{T}(\textbf{x}_{\textbf{k}'})|>\textbf{n}$, which is a contradiction.   

By Theorem \ref{thm: baire}, it follows that there exists a conditional ball $\textbf{B}_{\textbf{r}}(\textbf{z})$ on $1$ such that $\textbf{B}_{\textbf{r}}(\textbf{z})\sqsubset \textbf{E}_\textbf{n}$ for some $\textbf{n}\in\textbf{N}$.

Let $\textbf{x}\in \textbf{B}_\textbf{E}$ and $\textbf{T}\in \textbf{S}$. It follows
\[
|\textbf{T}(\textbf{x})| =\frac{\textbf{1}}{\textbf{r}}|\textbf{T}(\textbf{z}-\textbf{r} \textbf{x})- \textbf{T}(\textbf{z})|\leq \frac{\textbf{1}}{\textbf{r}}(|\textbf{T}(\textbf{z}-\textbf{r} \textbf{x})|+|\textbf{T}(\textbf{z})|)\leq \frac{\textbf{1}}{\textbf{r}}\textbf{2}\textbf{n}.
\]
\end{proof}

\subsection{Some results on conditional weak topologies}

We need to introduce the following notion:

\begin{defn}
Let $\textbf{E}$ be a conditional vector space. A conditional function $\textbf{p}:\textbf{E}\rightarrow\textbf{R}^+$ is a conditional seminorm if:
\begin{enumerate}
	\item $\textbf{p}(\textbf{r}\textbf{x})=|\textbf{r}|\textbf{p}(\textbf{x})$ for all $\textbf{r}\in\textbf{R}$ and $\textbf{x}\in\textbf{E}$;
	\item $\textbf{p}(\textbf{x}+\textbf{y})\leq\textbf{p}(\textbf{x})+\textbf{p}(\textbf{y})$ for all $\textbf{x},\textbf{y}\in\textbf{E}$.
\end{enumerate}
\end{defn}

\begin{defn}
Let $\mathcal{P}$ be a conditional family of conditional seminorms on a conditional vector space $\textbf{E}$. Given a conditionally finite subset $\textbf{Q}\sqsubset\mathcal{P}$ on $1$, and $\textbf{r}\in\textbf{R}^{++}$, we define
\[
\textbf{U}_{\textbf{Q},\textbf{r}}:=\left[\textbf{x}\in\textbf{E}\:;\: \underset{\textbf{p}\in \textbf{Q}}\nsup\textbf{p}(\textbf{x})\leq \textbf{r}\right].
\] 
Then, it follows by inspection that 
\[
\mathcal{U}:=\left[\textbf{x}+\textbf{U}_{\textbf{Q},\textbf{r}}\;;\:\textbf{x}\in\textbf{E},\:\textbf{r}\in\textbf{R}^{++},\: \textbf{Q}\sqsubset \mathcal{P}\textnormal{ on }1\textnormal{, conditionally finite} \right]
\]
is a conditional topological base. The conditional topology generated by $\mathcal{U}$ will be known as the conditional topology induced by $\mathcal{P}$, and $\textbf{E}$ endowed with this topology will be denoted by $\textbf{E}[\mathcal{P}]$.
\end{defn}

It can be checked that $\textbf{E}[\mathcal{P}]$ is a conditional locally convex space. 

\begin{defn}
Let $\textbf{E}$ be a conditional vector space and $\textbf{C}\sqsubset\textbf{E}$ on $1$. The conditional gauge functional is defined by
\[
\Vert\textbf{x}\Vert_{\textbf{C}}:=\ninf\left[ \textbf{r}\in\textbf{R}^{++}\:;\: \textbf{x}\in\textbf{r}\textbf{C}\right], \textnormal{ for any }\textbf{x}\in\textbf{E}.
\] 
\end{defn}

Then, we have the following result:
\begin{prop}
\label{prop: gauge}
Let $\textbf{E}$ be a conditional vector space and $\textbf{C}\sqsubset\textbf{E}$ on $1$, then:
\begin{enumerate}
	\item if $\textbf{C}$ is conditionally absorbing, then $\Vert\textbf{x}\Vert_{\textbf{C}}< +\infty$ for all $\textbf{x}\in\textbf{E}$;
	\item if $\textbf{C}$ is conditionally balanced, then $\Vert\textbf{r}\textbf{x}\Vert_{\textbf{C}}=|\textbf{r}|\Vert\textbf{x}\Vert_{\textbf{C}}$ for all $(\textbf{x},\textbf{r})\in\textbf{E}\Join\textbf{R}$;
	\item if $\textbf{C}$ is conditionally convex, then $\Vert\textbf{x}+\textbf{y}\Vert_{\textbf{C}}\leq \Vert\textbf{x}\Vert_{\textbf{C}}+\Vert\textbf{y}\Vert_{\textbf{C}}$ for all  $(\textbf{x},\textbf{y})\in\textbf{E}\Join\textbf{E}$.
\end{enumerate}
In particular, $\Vert\cdot\Vert_{\textbf{C}}$ is a conditional seminorm whenever $\textbf{C}$ is conditionally convex, conditionally absorbing, and conditionally balanced.
\end{prop}
\begin{proof}
\begin{enumerate}
	\item Let be given $\textbf{x}\in\textbf{E}$. If $\textbf{r}\in\textbf{R}^{++}$ with $\textbf{x}\in\textbf{r}\textbf{C}$, then $\textbf{p}_\textbf{C}(\textbf{x})\leq\textbf{r}$.
	\item Suppose $\textbf{x}\in\textbf{E}$ and $\textbf{r}\in\textbf{R}$. Let us put $a:=\supp(\textbf{r})$. 
		First, we have that $\Vert \textbf{r}\textbf{x}\Vert_\textbf{C}|a^c=\Vert\textbf{0}\Vert_\textbf{C}|a^c=\textbf{0}|a^c=\textbf{r}\Vert\textbf{x}\Vert_\textbf{C}|a^c$.
		
		Let us assume $a=1$ by consistently arguing on $a$. Thereby, $|\textbf{r}|>\textbf{0}$. Besides, for fixed $\textbf{s}\in\textbf{R}^{++}$, being $\textbf{C}$ conditionally balanced, we have that $\textbf{r}\textbf{x}\in\textbf{s}\textbf{C}$ if, and only if, $\textbf{x}\in\frac{\textbf{s}}{|\textbf{r}|}\textbf{C}$. Then, we have
\[
\Vert\textbf{r}\textbf{x}\Vert_\textbf{C}=\ninf\{\textbf{s}>\textbf{0}\:;\:\textbf{r}\textbf{x}\in\textbf{s}\textbf{C}\}=		
\ninf\{\textbf{s}>\textbf{0}\:;\:\textbf{x}\in\frac{\textbf{s}}{|\textbf{r}|}\textbf{C}\}=
\]
\[
|\textbf{r}|\ninf\{\frac{\textbf{s}}{|\textbf{r}|}\:;\:\textbf{s}>\textbf{0},\:\textbf{x}\in\frac{\textbf{s}}{|\textbf{r}|}\textbf{C}\}=
|\textbf{r}|\Vert\textbf{x}\Vert_\textbf{C}.
\]
\item For $\textbf{x},\textbf{y}\in\textbf{E}$, let $\textbf{r},\textbf{s}\in\textbf{R}^{++}$, with $\textbf{x}\in\textbf{r}\textbf{C}$ and $\textbf{y}\in\textbf{s}\textbf{C}$. Then, since $\textbf{C}$ is conditionally convex, we have
\[
\textbf{x}+\textbf{y}\in \textbf{r}\textbf{C} + \textbf{s}\textbf{C}=(\textbf{r}+\textbf{s})\left(\frac{\textbf{r}}{\textbf{r}+\textbf{s}}\textbf{C} +  \frac{\textbf{s}}{\textbf{r}+\textbf{s}}\textbf{C}\right)\sqsubset (\textbf{r}+\textbf{s})\textbf{C}.
\]
This implies that $\Vert\textbf{x}+\textbf{y}\Vert_\textbf{C}\leq\textbf{r}+\textbf{s}$. 

Now, by taking conditional infimums; first over $\textbf{r}$, and later over $\textbf{s}$, we obtain the result.

\end{enumerate}

\end{proof}

The conditional locally convex spaces can be characterized in a similar way as occurs in the classical case. More precisely speaking, a conditional topological space is conditionally locally convex if, and only if, there exists a conditional family of conditional seminorms inducing its conditional topology. The proof does not has any surprising element and can be easily written following the non conditional case, and employing the conditional gauge function. 

\begin{rem}
Some comments can be made in this regard. Before the conditional set theory was introduced, Filipovic et al. \cite{key-10} tried to state a similar characterization in the related context of topological $L^0$-modules (see \cite[Theorem 2.4]{key-10}). However, as was pointed out Guo et al.\cite{key-34}, there is some lack of precision in ascertaining which stability properties are required for this statement to be true; moreover, neither the algebraic structure nor the topological structure of a general $L^0$-module need to be stable with respect to the underlying measure algebra. In fact, Zapata \cite{key-17}, and independently Wu and Guo \cite{key-18}, provided a counterexample to this statement, and proved that a mild type of stability property is required on the elements of the neighborhood base of $0$ considered in Theorem 2.4 of \cite{key-10}.
\end{rem}

An example of conditional topology induced by a conditional family of conditional seminorms are the conditional weak topologies. For a  conditional locally convex space $\textbf{E}[\mathcal{T}]$, let $\textbf{E}[\mathcal{T}]^*$ denote $-$or simply $\textbf{E}^*$$-$, the conditional vector space of conditionally linear and conditionally continuous functions $\textbf{f}:\textbf{E}\rightarrow\textbf{R}$. We define the conditional weak topology $\sigma(\textbf{E},\textbf{E}^*)$ on $\textbf{E}$ as the conditional locally convex topology induced by the conditional family of conditional seminorms $\left[\textbf{p}_{\textbf{x}^*}\:;\:\textbf{x}^*\in \textbf{E}^* \right]$ defined by $\textbf{p}_{\textbf{x}^*}=|\textbf{x}^*(\textbf{x})|$ for $\textbf{x}\in\textbf{E}$. Analogously, the conditional weak-$*$ topology $\sigma(\textbf{E}^*,\textbf{E})$ on $\textbf{E}^*$ is defined.

\begin{rem}
\label{rem: weakT}
Conditional dual pairs are introduced and discussed in \cite{key-7} (see \cite[Definition 5.6]{key-7}), also a more detailed discussion can be found in \cite[Section 4]{key-35}.  Further, in \cite{key-7} it was shown that the pairing $\langle \textbf{E}, \textbf{E}^*\rangle$  together with conditionally bilinear form $\langle , \rangle:\textbf{E}\Join\textbf{E}^*\rightarrow\textbf{R}$ defined by $\langle\textbf{x},\textbf{x}^*\rangle:=\textbf{x}^*(\textbf{x})$ is a dual pair. Consequently we have $\textbf{E}^*\left[\sigma(\textbf{E}^*,\textbf{E})\right]^*=\textbf{E}$.  
\end{rem}

Now let us turn to see some results related to weak topologies.

\begin{prop}
\label{prop: weakClosure}
Let $\textbf{E}[\mathcal{T}]$ be a conditional locally convex space, and let $\textbf{A}\sqsubset\textbf{E}$ be on $1$  conditionally convex. Then $\overline{\textbf{A}}^{\mathcal{T}}=\overline{\textbf{A}}^{\sigma(\textbf{E},\textbf{E}^*)}$.
\end{prop}
\begin{proof}

Clearly, $\overline{\textbf{A}}^{\mathcal{T}}\sqsubset\overline{\textbf{A}}^{\sigma(\textbf{E},\textbf{E}^*)}$.

For the reverse conditional inclusion, let us take $\textbf{x}\in\overline{\textbf{A}}^{\sigma(\textbf{E},\textbf{E}^*)}$.

Let us consider $b:=\vee\left\{ a\in\A \:;\: x|a\in\overline{\textbf{A}}^{\mathcal{T}}\right\}$.

Since $b$ is attained, if $b=1$ we get the result. If not, we can consistently argue on $b^c$, and suppose w.l.g. $b=0$.

If so, we have that $[\textbf{x}]\sqcap\overline{\textbf{A}}^{\mathcal{T}}=\textbf{E}|0$. Besides, $[\textbf{x}]$ is conditionally compact, and $\overline{\textbf{A}}^{\mathcal{T}}$ is conditionally convex and conditionally closed with respect to $\mathcal{T}$.

Then, Theorem \ref{thm: HB2} yields $\textbf{x}^*\in \textbf{E}^*$ and $\textbf{r}\in\textbf{R}^{++}$ such that 
\begin{equation}
\label{eq: mazur}
\begin{array}{cc}
\textbf{x}^*(\textbf{x})+\textbf{r}\leq\textbf{x}^*(\textbf{y}), & \textnormal{ for all }\textbf{y}\in\overline{\textbf{A}}^{\mathcal{T}}.
\end{array}
\end{equation}
Let us define $\textbf{B}:=\left[ y\in\textbf{E} \:;\: \textbf{x}^*(\textbf{y})\geq\textbf{x}^*(\textbf{x})+\textbf{r} \right]$, which is a conditionally $\sigma(\textbf{E},\textbf{E}^*)$-closed subset. 

But (\ref{eq: mazur}) means that $\textbf{A}\sqsubset \textbf{B}$ and $\textbf{x}\in \textbf{B}^{\sqsubset}$. This is a contradiction as $\textbf{x}$ is in the conditional weak closure of $\textbf{A}$.  
\end{proof}

\begin{rem}
The above result is a conditional version of the classical Mazur's theorem. In literature, we can find some related results. For instance, it can be found a version of this theorem for $L^0$-normed modules, but considering mild stability properties, in \cite[Corollary 2.1]{key-15}; a version of Mazur's theorem for $L^\infty$-modules is proved in \cite[Theorem 10.1]{key-14}; and another version for $L^0$-modules, but with the topology of stochastic convergence with respect to a family of $L^0$-seminorms, was provided in \cite[Corollary 3.4]{key-37}. 
\end{rem}

In the following result we give the \textit{natural conditional embedding} of a conditionally normed space $\textbf{E}$ into the conditional second dual $\textbf{E}^{**}:=(\textbf{E}^{*})^{*}$.

\begin{thm}
\label{thm: imbedding}
Let $\textbf{E}$ be a conditionally normed space. For each $\textbf{x}\in\textbf{E}$, let us define the conditional function $\textbf{T}_\textbf{x}:\textbf{E}^*\rightarrow\textbf{R}$ with $\textbf{T}_\textbf{x}(\textbf{x}^*):=\textbf{x}^*(\textbf{x})$. Then $\textbf{T}_\textbf{x}\in\textbf{E}^{**}$ for each $\textbf{x}\in\textbf{E}$, and the conditional function $\textbf{j}:\textbf{E}\rightarrow\textbf{E}^{**}$ given by $\textbf{j}(\textbf{x}):=\textbf{T}_\textbf{x}$ is a conditional isometry, i.e. $\Vert \textbf{x}\Vert=\Vert\textbf{j}(\textbf{x})\Vert$ for all $\textbf{x}\in\textbf{E}$.  
\end{thm}
\begin{proof}
It suffices to show that for all $\textbf{x}\in\textbf{E}$ it holds that $\Vert\textbf{x}\Vert=\nsup\left[|\textbf{x}^{*}(\textbf{x})|\:;\:\textbf{x}^{*}\in\textbf{B}_{\textbf{E}^{*}}\right]$.

First, let us show that there exists $\textbf{x}^{*}\in\textbf{B}_{\textbf{E}^*}$ such that $\textbf{x}^{*}(\textbf{x})=\Vert\textbf{x}\Vert$.
 Indeed, we have that the conditional function $\textbf{x}^*_0:\spa_\textbf{R} [\textbf{x}]\rightarrow\textbf{R}$ with $\textbf{x}^*_0(\textbf{r}\textbf{x}):=\textbf{r}\Vert\textbf{x}\Vert$ is well defined and $\textbf{x}^*_0\in(\spa_\textbf{R} [\textbf{x}])^*$. 

Second, notice that $\textbf{x}^*_0(\textbf{y})\leq \Vert\textbf{y}\Vert$ for all $\textbf{y}\in\spa_\textbf{R} [\textbf{x}]$. By Theorem \ref{thm: HB}, there exists $\textbf{x}^*\in\textbf{E}^*$ which extends $\textbf{x}^*_0$ and such that $|\textbf{x}^*(\textbf{y})|\leq \Vert\textbf{y}\Vert$ for all $\textbf{y}\in\textbf{E}$. We conclude that $\Vert\textbf{x}^*\Vert\leq \textbf{1}$. Moreover, notice that, in fact, $\Vert\textbf{x}^*\Vert=\textbf{1}$.

Finally, let $\textbf{r}:=\nsup\left[|\textbf{z}^{*}(\textbf{x})|\:;\:\textbf{z}^{*}\in\textbf{B}_{\textbf{E}^{*}}\right]$. Then it is clear that $\textbf{r}\leq \Vert\textbf{x}\Vert=\textbf{x}^*(\textbf{x})\leq |\textbf{x}^*(\textbf{x})|\leq\textbf{r}$, and we obtain what was asserted.   
\end{proof}


\begin{thm}
\label{thm: Goldstine}
[Conditional version of Goldstine's theorem]
Let $\textbf{j}:\textbf{E}\rightarrow\textbf{E}^{**}$ be the natural conditional embedding. Then $\textbf{j}(\textbf{B}_\textbf{E})$ is conditionally $\sigma(\textbf{E}^{**},\textbf{E}^*)$-dense in $\textbf{B}_{\textbf{E}^{**}}$.  
\end{thm}
\begin{proof}
Define $\textbf{K}:=\overline{\textbf{j}(\textbf{B}_\textbf{E})}^{\sigma(\textbf{E}^{**},\textbf{E}^*)}$. Since $\textbf{B}_\textbf{E}^{**}$ is conditionally $\sigma(\textbf{E}^{**},\textbf{E}^{*})$-closed, it follows that $\textbf{K}\sqsubset\textbf{B}_\textbf{E}^{**}$. Besides, $\textbf{K}$ is conditionally convex; then, if we suppose that there exists $\textbf{x}^{**}\in\textbf{K}^{\sqsubset}\sqcap\textbf{B}_{\textbf{E}^{**}}$, Theorem \ref{thm: HB2} provides us with a conditionally $\sigma(\textbf{E}^{**},\textbf{E}^*)$-continuous linear function $\textbf{f}:\textbf{E}^{**}\rightarrow\textbf{R}$ so that
\[
\nsup\left[\textbf{f}(\textbf{s})\:;\: \textbf{s}\in \textbf{K}\right]=\textbf{x}^{**}<\textbf{f}(\textbf{x}^{**}).
\]
In addition, in view of Remark \ref{rem: weakT}, we have  $\textbf{E}^{**}[\sigma(\textbf{E}^{**},\textbf{E}^*)]^*=\textbf{E}^*$.  Consequently, we  can find $\textbf{x}^*\in \textbf{E}^{*}$ with $\textbf{f}(\textbf{z}^{**})=\textbf{z}^{**}(\textbf{x}^*)$ for all $\textbf{z}^{**}\in\textbf{E}^{**}$.

Then, being $\textbf{B}_\textbf{E}$ conditionally balanced, it follows that 
\begin{equation}
\label{eq: Gold}
\nsup\left[\textbf{x}^{*}(\textbf{x})\:;\: \textbf{x}\in \textbf{B}_\textbf{E}\right]=\nsup\left[|\textbf{x}^{*}(\textbf{x})|\:;\: \textbf{x}\in \textbf{B}_\textbf{E}\right]< \textbf{f}(\textbf{x}^{**})=\textbf{x}^{**}(\textbf{x}^*).
\end{equation}
But this is impossible given that 
\[
\textbf{x}^{**}(\textbf{x}^*)=|\textbf{x}^{**}(\textbf{x}^*)|\leq \Vert \textbf{x}^{**}\Vert\Vert \textbf{x}^*\Vert < \textbf{x}^{**}(\textbf{x}^*)\Vert \textbf{x}^{**}\Vert\leq \textbf{x}^{**}(\textbf{x}^*).
\] 
Notice that the first equality is true because the conditional supremum (\ref{eq: Gold}) is conditionally greater than or equal to $\textbf{0}$.  
\end{proof}

\section{Conditional versions of Eberlein-\v{S}mulian and Amir-Lindenstrauss theorems}
\label{sec3}

Finally, in this section we will prove the main results of the paper. Let us state the first one:

\begin{thm} 
\label{thm: EberleinSmulian}[Conditional Eberlein-\v{S}mulian theorem]
A conditional subset of a conditionally  $\textbf{E}$ normed space is conditionally weakly compact if, and only if, it is conditionally weakly sequentially compact.
\end{thm}

The other important result drawn from classical study of weak topologies and extended to the framework of conditional sets, is the Amir-Lindenstrauss theorem. Conditional Banach-Alaoglu's theorem (see \cite[Theorem 5.10]{key-7})  ensures that every conditionally bounded sequence in the conditional dual of a conditionally normed space has a conditionally weakly-$*$ convergent subnet. The conditional version of Amir-Lindenstrauss theorem gives sufficient conditions that ensure that one can extract a conditionally weakly-$*$ convergent subsequence. 

Before stating this theorem we must introduce the following notion:   

\begin{defn}
A conditional Banach space $\textbf{E}$ is said to be conditionally weakly compactly generated if there is a conditional subset $\textbf{C}$ of $\textbf{E}$ on $1$ which is conditionally weakly compact, conditionally convex and conditionally balanced,  such that $\overline{\spa_\textbf{R}\textbf{C}}=\textbf{E}$.
\end{defn}

\begin{thm}\label{thm: AmirLindenstrauss}[Conditional Amir-Lindenstrauss theorem]  Every  conditionally weakly compactly generated Banach space has a conditionally weakly-$*$ sequentially compact dual unit ball.
\end{thm}

Before proving these two theorems, we need to make a remark. We claim that $1=\supp(\textbf{E}^*)=\supp(\textbf{E}^{**})=...$ Indeed, we assumed $\supp(\textbf{E})=1$. Then, we can choose $\textbf{x}\in [\textbf{0}]^\sqsubset$ with $\supp(\textbf{x})=1$. Reasoning as in the proof of Theorem \ref{thm: imbedding}, we can find $\textbf{x}^*\in\textbf{E}^*$ such that $\textbf{x}^*(\textbf{x})=\Vert\textbf{x}\Vert$. In particular $\textbf{x}^*\in[\textbf{0}]^\sqsubset$. This proves that $\supp(\textbf{E}^*)=1$. The same applies to $\textbf{E}^{**}=(\textbf{E}^*)^*$, and so on.

Now, let us turn to show some preliminary results:

\begin{lem}
\label{lem: L0bounded}
Let $\textbf{E}$ be a conditionally normed space. Suppose that $\textbf{K}\sqsubset\textbf{E}$ is conditionally weakly compact, then $\textbf{K}$ is conditionally bounded and conditionally norm closed. 
\end{lem}
\begin{proof}
Let us suppose w.l.g. $\textbf{K}$ on $1$. If $\textbf{x}^*\in \textbf{E}^*$, then $\textbf{x}^*$ is a conditionally weakly continuous function, and $\textbf{x}^*(\textbf{K})$ is therefore conditionally weakly compact (see \cite[Proposition 3.26]{key-7}), hence it is a conditionally bounded subset of $\textbf{R}$. Since $\textbf{K}$ can be conditionally embedded into $\textbf{E}^{**}$ (see Theorem \ref{thm: imbedding}), we can apply Theorem \ref{thm: uniformBoundeness} obtaining that $\textbf{K}$ is conditionally bounded.   
\end{proof}

\begin{lem}
\label{lem: WSClosure}
Let $\textbf{E}$ be a conditionally normed space and suppose $\textbf{K}\sqsubset \textbf{E}$ on $1$, conditionally bounded and conditionally weakly closed. Then, $\overline{\textbf{j}(\textbf{K})}^{\sigma(\textbf{E}^{**},\textbf{E}^*)}\sqsubset \textbf{j}(\textbf{E})$, where $\textbf{j}$ is the natural conditional embedding  if, and only if, $\textbf{K}$ is conditionally weakly compact.
\end{lem}
\begin{proof} 
Suppose that $\overline{\textbf{j}(\textbf{K})}^{\sigma(\textbf{E}^{**},\textbf{E}^*)}\sqsubset \textbf{j}(\textbf{E})$. By the conditional theorem of Banach-Alaoglu (see \cite[Theorem 5.10]{key-7}) we have that $\overline{\textbf{j}(\textbf{K})}^{\sigma(\textbf{E}^{**},\textbf{E}^*)}$ is conditionally $\sigma(\textbf{E}^{**},\textbf{E}^*)$-compact. Since $\sigma(\textbf{E}^{**},\textbf{E}^*)|_{\textbf{j}(\textbf{E})}=\sigma(\textbf{j}(\textbf{E}),\textbf{E}^*)$, it follows that
\[
\textbf{j}(\textbf{K})=\overline{\textbf{j}(\textbf{K})}^{\sigma(\textbf{j}(\textbf{E}),\textbf{E}^*)}=\overline{\textbf{j}(\textbf{K})}^{\sigma(\textbf{E}^{**},\textbf{E}^*)}\sqcap \textbf{j}(\textbf{E}) = \overline{\textbf{j}(\textbf{K})}^{\sigma(\textbf{E}^{**},\textbf{E}^*)}.
\]
We conclude that $\textbf{K}$ is conditionally weakly compact.

\end{proof}

\begin{lem}
\label{lem: totalSet}
Assume that $\textbf{E}$ is conditionally Banach and conditionally separable. Then, there is a conditionally countable total subset $\textbf{D}$ of $\textbf{E}^*$.
\end{lem}
\begin{proof}
First of all note that $\textbf{C}$ is on $1$, since it is conditionally dense.

For each $\textbf{x}\in \textbf{C}$, as we argued in the proof of Theorem \ref{thm: imbedding}, there exists some $\textbf{z}^*_{\textbf{x}}\in\textbf{B}_{\textbf{E}^*}$  such that $\textbf{z}^*(\textbf{x})=\Vert \textbf{x} \Vert$. Thereby, for each $\textbf{x}\in \textbf{C}$, we can consider the conditional set $\textbf{A}_\textbf{x}:=\left[\textbf{z}^*\in\textbf{B}_{\textbf{E}^*}\:;\:\textbf{z}^*(\textbf{x})=\Vert \textbf{x} \Vert\right]$, which is well defined and on $1$. Further, $\{\textbf{A}_\textbf{x}\}_{\textbf{x}\in\textbf{C}}$ is a conditional family of subsets of $\textbf{B}_{\textbf{E}^*}$. By the  conditional version of the axiom of choice (see \cite[Theorem 2.26]{key-7}), we can find a conditional family $\{\textbf{z}_{\textbf{x}}^*\}_{\textbf{x}\in\textbf{C}}$ with $\textbf{z}^*_\textbf{x}\in\textbf{A}_\textbf{x}$ for all $\textbf{x}\in\textbf{C}$.

Let us define the conditional subset $\textbf{D}:=\left[\textbf{z}_{\textbf{x}}^* \:;\: \textbf{x}\in \textbf{C}\right]$. 

Fix $\textbf{x}\in\textbf{E}$, $\textbf{x}\neq\textbf{0}$. By arguing on $\supp(\textbf{x})$, which is not null, we can suppose that $\textbf{x}\in\textbf{E}\sqcap[\textbf{0}]^\sqsubset$. Since $\textbf{C}$ is conditionally dense, there exists a conditional sequence $\{\textbf{x}_\textbf{n}\}$ in $\textbf{C}$ which conditionally converges to $\textbf{x}$, and thus $\nlim_{\textbf{n}} \textbf{z}^*_{\textbf{x}_\textbf{n}} (\textbf{x}) = \nlim_{\textbf{n}}[\textbf{z}^*_{\textbf{x}_\textbf{n}}(\textbf{x}-\textbf{x}_\textbf{n})+\textbf{z}^*_{\textbf{x}_\textbf{n}}(\textbf{x}_\textbf{n})]=\nlim_{\textbf{n}} \Vert \textbf{x}_\textbf{n} \Vert=\Vert\textbf{x}\Vert>\textbf{0}$. Necessarily, there exists $\textbf{n}_0\in\textbf{N}$ such that $\textbf{z}^*_{\textbf{x}_{\textbf{n}_0}} (\textbf{x})>\textbf{0}$. We conclude that $\textbf{D}$ is conditionally total.       
\end{proof}

\begin{lem}
\label{lem: weakBounded}
If $\textbf{E}$ is a conditionally normed space, and $\{\textbf{x}_\textbf{n}\}_{\textbf{n}\in\textbf{N}}$ is a conditional sequence which conditionally weakly converges to $\textbf{x}\in\textbf{E}$, then $\{\textbf{x}_\textbf{n}\}_{\textbf{n}\in\textbf{N}}$ is conditionally bounded. As a consequence, if $\textbf{K}$ is conditionally weakly sequentially compact, then $\textbf{K}$ is conditionally bounded. 
\end{lem}
\begin{proof}
For the first part, we consider the conditional sequence $\{\textbf{z}_\textbf{n}\}_{\textbf{n}\in\textbf{N}}$ with $\textbf{z}_\textbf{n}:=\textbf{j}(\textbf{x}_\textbf{n})$ for every $\textbf{n}\in\textbf{N}$, where $\textbf{j}$ is the natural conditional embedding. This is a conditional sequence in $\textbf{E}^{**}$.  Since $\{\textbf{x}_\textbf{n}\}$ conditionally converges to $\textbf{x}$, for fixed $\textbf{x}^*\in\textbf{E}^*$, there is $\textbf{n}_0\in \textbf{N}$ so that $|\textbf{x}^*(\textbf{x}_\textbf{n})|\leq \textbf{1}+ |\textbf{x}^*(\textbf{x})|$ whenever $\textbf{n}\geq \textbf{n}_0$. Then, for each $\textbf{n}\in\textbf{\textbf{N}}$ 
\[
|\textbf{z}_\textbf{n}(\textbf{x}^*)|=|\textbf{x}^*(\textbf{x}_\textbf{n})|\leq (\textbf{1} + |\textbf{x}^*(\textbf{x})|) \vee \nmax \{ |\textbf{x}^*(\textbf{x}_\textbf{n})|\:;\: \textbf{n}\leq \textbf{n}_0\}.
\]
 
Now, due to Theorem \ref{thm: uniformBoundeness}, we have that $\{\textbf{z}_\textbf{n}\}$ is conditionally bounded.  Consequently $\{\textbf{x}_\textbf{n}\}$ is conditionally bounded too, because $\textbf{j}$ is a conditional isometry. 

As for the second part, we can suppose without loss of generality that $\textbf{K}$ is on $1$. Define
\[
d:=\vee\left\{a\in\A\:;\: \textbf{K}|a \textnormal{ is conditionally bounded }  \right\}.
\] 
$d$ is attained. So if $d=1$ we are done. In other case, we can assume that $d=0$. 

If so, for each $n\in\N$ let us define
\[
d_n:=\vee\left\{a\in\A\:;\: \exists\textbf{x}\in\textbf{K}\textnormal{ with } \Vert\textbf{x}\Vert|a>\textbf{n}|a\right\}.
\]
Since $d=0$, necessarily $d_n=1$ and it is attained, so we can pick $\textbf{x}_n\in \textbf{K}$ such that $\Vert \textbf{x}_{n}\Vert>\textbf{n}$. For $\textbf{n}\in\textbf{N}$ with $n=\sum n_i|a_i\in{N}$, define $\textbf{x}_\textbf{n}:=x_n|1$ with $x_n:=\sum x_{n_i}|a_i$. Then $\{\textbf{x}_\textbf{n}\}$ is a conditional sequence with $\Vert\textbf{x}_\textbf{n}\Vert>\textbf{n}$ for each $\textbf{n}\in\textbf{N}$. Due to the first part of the theorem, we find that the conditional sequence $\{\textbf{x}_\textbf{n}\}$ cannot have a conditionally weakly convergent subsequence. 
\end{proof}

\begin{lem}
\label{lem: metrizable}
Let $\textbf{E}$ be a conditionally normed space. If $\textbf{K}\sqsubset \textbf{E}$ is conditionally weakly compact and there is a conditionally countable set $\textbf{D}\sqsubset \textbf{E}^*$ which is conditionally total, and such that $\Vert \textbf{x}^*\Vert>\textbf{0}$ for all $\textbf{x}^*\in\textbf{D}$, then $\textbf{K}$ is conditionally metrizable.

\end{lem}
\begin{proof}

Since $\textbf{D}$ is conditionally countable, we can choose a conditional sequence $\{\textbf{x}^*_\textbf{n}\}$ with $\textbf{D}=\left[\textbf{x}^*_\textbf{n} \:;\: \textbf{n}\in\textbf{N}\right]$. We will construct a conditional metric on $\textbf{K}$. Indeed, we define the conditional function $\textbf{d}: \textbf{E}\Join \textbf{E}\rightarrow \textbf{R}$ given by 
\[
\begin{array}{cc}
\textbf{d}(\textbf{x},\textbf{y}):=\sum_{\textbf{n}\geq\textbf{1}} \frac{\textbf{1}}{\textbf{2}^\textbf{n}}\frac{|\textbf{x}^*_\textbf{n}(\textbf{x}-\textbf{y})|}{\Vert \textbf{x}^*_\textbf{n}\Vert} & \textnormal{for }(\textbf{x},\textbf{y})\in \textbf{E}\Join \textbf{E}.
\end{array}
\]
By using that $\textbf{K}$ is conditionally total, it can be verified that $\textbf{d}$ is a conditional metric on $\textbf{E}$. Besides, due to Lemma \ref{lem: L0bounded}, we have that $\textbf{K}$ is conditionally bounded. Now, we consider the identity $\textbf{Id}:\textbf{K}\rightarrow\textbf{K}$ as a conditional function between the conditional metric spaces $(\textbf{K},\sigma(\textbf{E},\textbf{E}^*)|_\textbf{K})$ and $(\textbf{K},\textbf{d})$. Then $\textbf{Id}$ is conditionally continuous. 
Indeed, let $\{\textbf{x}_\textbf{i}\}_{\textbf{i}\in\textbf{I}}$ be a conditional net in $\textbf{K}$ which conditionally converges  to some $\textbf{x}\in \textbf{K}$ with respect to $\sigma(\textbf{E},\textbf{E}^*)|_\textbf{K}$. For arbitrary $\textbf{r}\in\textbf{R}^{++}$, due to the conditional boundedness of $\textbf{K}$, there exists $\textbf{m}\in \textbf{N}$, $\textbf{m}>\textbf{1}$, and $\textbf{r}\in\textbf{R}^{++}$ so that
\[
\sum_{\textbf{n}\geq \textbf{m}} \frac{\textbf{1}}{\textbf{2}^\textbf{n}}\frac{|\textbf{x}^*_\textbf{n}(\textbf{x}_\textbf{i} - \textbf{x})|}{\Vert \textbf{x}^*_\textbf{n}\Vert}\leq \frac{\textbf{1}}{\textbf{2}^\textbf{n}} \nsup\left[\Vert\textbf{x}\Vert\:;\: \textbf{x}\in\textbf{K}\right]  \leq \textbf{r}/\textbf{2},  
\] 
for all $\textbf{i}\in \textbf{I}$. Since $\{\textbf{x}_\textbf{i}\}$ conditionally converges to $\textbf{x}$, we can also choose $\textbf{i}_0\in\textbf{I}$ such that   
\[
\sum_{\textbf{1}\leq \textbf{n}\leq\textbf{m}-1} \frac{\textbf{1}}{\textbf{2}^\textbf{n}}\frac{|\textbf{x}^*_\textbf{n}(\textbf{x}_\textbf{i} - \textbf{x})|}{\Vert \textbf{x}^*_\textbf{n}\Vert}\leq \textbf{r}/\textbf{2}, 
\]
for all $\textbf{i}\geq \textbf{i}_0$. It follows that $\textbf{d}(\textbf{x}_\textbf{i},\textbf{x})\leq \textbf{r}$ whenever $\textbf{i}\geq \textbf{i}_0$.

Finally, let us show that $\textbf{Id}$ is a conditionally open function.

Let $\textbf{U}\sqsubset\textbf{K}$ be conditionally $\sigma(\textbf{E},\textbf{E}^*)|_\textbf{K}$-open. Since $\textbf{K}$ is conditionally $\sigma(\textbf{E},\textbf{E}^*)$-compact, it follows that $\textbf{U}^{\sqsubset}\sqcap\textbf{K}$ is conditionally $\sigma(\textbf{E},\textbf{E}^*)|_\textbf{K}$-compact. Then, since $\textbf{Id}$ is conditionally continuous, in view of \cite[Proposition 3.26]{key-7}, $\textbf{U}^{\sqsubset}\sqcap\textbf{K}$ is conditionally compact in $(\textbf{K},\textbf{d})$, hence $\textbf{U}$ is conditionally open in $(\textbf{K},\textbf{d})$.
\end{proof}

Let us turn to prove Theorem \ref{thm: EberleinSmulian}.
\begin{proof}

Suppose that $\textbf{K}$ is a conditionally weakly compact subset of $\textbf{E}$ on $1$, and let $\{\textbf{z}_\textbf{n}\}_{\textbf{n}\in \textbf{N}}$ be a conditional sequence in $\textbf{K}$. Let us put $\textbf{K}_0:=\overline{\spa_{\textbf{R}}[\textbf{z}_\textbf{n}\:;\:\textbf{n}\in \textbf{N}]}^{\sigma(\textbf{E},\textbf{E}^*)}$. Then $\textbf{K}_0$ is conditionally weakly closed in $\textbf{E}$ and, consequently, $\textbf{K}_0\sqcap \textbf{K}$ is  conditionally $\sigma(\textbf{E},\textbf{E}^*)$-compact in the conditionally complete normed space $\textbf{K}_0$. Such space is conditionally separable, because it conditionally contains the conditionally countable and dense subset $\spa_{\textbf{Q}}[\textbf{z}_\textbf{n}\:;\:\textbf{n}\in \textbf{N}]$. We can thus apply Lemma \ref{lem: totalSet}; that is, $\textbf{K}_0$ conditionally contains a conditionally countable total set.

From Lemma \ref{lem: metrizable} we know that $\textbf{K}_0\sqcap \textbf{K}$ is conditionally metrizable with the conditional topology $\sigma(\textbf{K}_0,\textbf{E}^*)$. It is known that conditional compactness and conditional sequential compactness are equivalent in a conditional metric space (see \cite[Theorem 4.6]{key-7}). Then $\textbf{K}_0\sqcap \textbf{K}$ is conditionally sequentially $\sigma(\textbf{K}_0,\textbf{E}^*)$-compact. In particular, if we choose a conditional subsequence $\{\textbf{z}_{\textbf{n}_\textbf{m}}\}_{\textbf{m}\in\textbf{N}}$ which conditionally converges in $\sigma(\textbf{K}_0,\textbf{E}^*)$ to $\textbf{z}$, then it also conditionally  converges in $\sigma(\textbf{E},\textbf{E}^*)$ to $\textbf{z}$. 

We now turn to the converse. Let us first make an observation: if $\textbf{F}$ is a conditionally finitely generated vector subspace of $\textbf{E}^{**}$, then there exists a conditionally finite subset $[\textbf{z}^*_\textbf{n}\:;\: \textbf{n}\leq \textbf{m}]$ of $\textbf{S}_{\textbf{E}^*}$ such that for any $\textbf{x}^{**}\in \textbf{F}$
\begin{equation}
\label{ineq0}
\frac{\Vert \textbf{x}^{**}\Vert}{\textbf{2}}\leq \nmax \left[ |\textbf{x}^{**}(\textbf{z}^*_\textbf{n})| \:;\: \textbf{n}\leq \textbf{m}\right].
\end{equation}

Indeed, since $\textbf{S}_\textbf{F}$ is conditionally $\Vert\cdot\Vert$-compact in $\textbf{F}$ (Proposition \ref{prop: Heine-Borel}), there exists a conditional finite subset $[\textbf{z}^{**}_\textbf{n}\:;\: \textbf{n}\leq \textbf{m}]$ of $\textbf{S}_\textbf{F}$ such that
\begin{equation}
\label{eq: sphere}
\textbf{S}_\textbf{F}\sqsubset\underset{\textbf{1}\leq\textbf{n}\leq \textbf{m}} \sqcup\textbf{B}_{\textbf{1}/\textbf{4}}(\textbf{z}^{**}_\textbf{n}).
\end{equation}

For each $n\in\N$, let us put $a_n:=\vee\left\{b\in\A\:;\:\textbf{n}|b\leq\textbf{m}|b\right\}$, and define $l=n|a_n + 1|a_n^c$. Since $\Vert \textbf{z}^{**}_\textbf{l}\Vert=\textbf{1}$, we can choose $\textbf{z}^*_n \in \textbf{S}_{\textbf{E}^*}$ so that $\textbf{z}^{**}_n (\textbf{z}^*_\textbf{l})\geq \frac{\textbf{3}}{\textbf{4}}$. Now, for $\textbf{n}\in\textbf{N}$ with $\textbf{n}\leq\textbf{m}$, where $n=\sum n_i|a_i$, we define $\textbf{z}^*_\textbf{n}:=z^*_n|1$ with $z^*_n:=\sum z^*_{n_i}|a_i$. Then, the conditional family $[\textbf{z}^{*}_\textbf{n} \:;\:  \textbf{1}\leq\textbf{n}\leq \textbf{m}]$ satisfies the condition $\textbf{z}^{**}_\textbf{n} (\textbf{z}^*_\textbf{n})\geq \frac{\textbf{3}}{\textbf{4}}$.

Now, in view of (\ref{eq: sphere}), we have that for each $\textbf{x}^{**}\in\textbf{S}_\textbf{F}$ there is $\textbf{n}\leq\textbf{m}$ such that
\begin{equation}
\label{ineq}
\textbf{x}^{**}(\textbf{z}^*_\textbf{n})= \textbf{z}^{**}_\textbf{n}(\textbf{z}^*_\textbf{n}) - (\textbf{z}^{**}_\textbf{n}-\textbf{x}^{**})(\textbf{z}^*_\textbf{n})  \geq  \frac{\textbf{1}}{\textbf{2}}. 
\end{equation}  

For given $\textbf{x}^{**}\in\textbf{F}$, put $a:=\supp(\textbf{x}^{**})$. Then $\textbf{x}^{**}{\Vert \textbf{x}^{**}\Vert}^{-1}|a\in\textbf{S}_\textbf{F}|a$, and due to (\ref{ineq}), we obtain that $\textbf{x}^{**}(\textbf{z}^*_\textbf{n})|a\geq\Vert \textbf{x}^{**}\Vert/\textbf{2}|a$ for some $\textbf{n}\leq \textbf{m}$. Therefore, (\ref{ineq0}) follows.

Let $\textbf{K}\sqsubset\textbf{E}$ be on $1$ and conditionally weakly sequentially compact. Consider $\overline{\textbf{j}(\textbf{K})}^{\sigma(\textbf{E}^{**},\textbf{E}^*)}$, where $\textbf{j}:\textbf{E}\rightarrow\textbf{E}^{**}$ is the natural conditional embedding. By Lemma \ref{lem: weakBounded}, $\textbf{K}$ is conditionally bounded, and therefore $\overline{\textbf{j}(\textbf{K})}^{\sigma(\textbf{E}^{**},\textbf{E}^*)}$ is conditionally bounded as well.

We will show $\overline{\textbf{j}(\textbf{K})}^{\sigma(\textbf{E}^{**},\textbf{E}^*)}$ actually lies in $\textbf{j}(\textbf{E})$, which will yield that $\textbf{K}$ is conditionally weakly compact, in view of Lemma \ref{lem: WSClosure}.

  Indeed, let us take $\textbf{x}^{**}\in \overline{\textbf{j}(\textbf{K})}^{\sigma(\textbf{E}^{**},\textbf{E}^*)}$, and choose $\textbf{x}_1^*\in \textbf{S}_{\textbf{E}^*}$. Since $\textbf{x}^{**}\in \overline{\textbf{j}(\textbf{K})}^{\sigma(\textbf{E}^{**},\textbf{E}^*)}$ we have that there exists $\textbf{z}_1 \in \textbf{K}$ such that 
	\[
	|(\textbf{x}^{**}-\textbf{j}(\textbf{z}_1))(\textbf{x}^*_1)|\leq \textbf{1}.
	\]
	
	$\spa_{\textbf{R}}[\textbf{x}^{**},\textbf{x}^{**}-\textbf{j}(\textbf{z}_1)]$ is a conditionally finitely generated vector subspace of $\textbf{E}^{**}$. The previous observation provides us with a conditionally finite subset $[\textbf{x}^*_\textbf{n}\:;\:\textbf{1}< \textbf{n}\leq \textbf{n}_2]\sqsubset\textbf{S}_{\textbf{E}^*}$ such that for any $\textbf{y}^{**}\in\spa_{\textbf{R}}[\textbf{x}^{**},\textbf{x}^{**}-\textbf{j}(\textbf{z}_1)]$,
\[
\frac{\Vert \textbf{y}^{**} \Vert}{\textbf{2}} \leq \nmax \left[ |\textbf{y}^{**}(\textbf{x}_\textbf{n})| \:;\: \textbf{1}< \textbf{n}\leq \textbf{n}_2\right] \leq \nmax \left[ |\textbf{y}^{**}(\textbf{x}_\textbf{n})| \:;\: \textbf{1}\leq \textbf{n}\leq \textbf{n}_2\right].
\]		

Since $\textbf{x}^{**}\in \overline{\textbf{j}(\textbf{K})}^{\sigma(\textbf{E}^{**},\textbf{E}^*)}$, we can choose $\textbf{z}_2\in\textbf{K}$ such that  
\[
\begin{array}{cc}
|(\textbf{x}^{**}-\textbf{j}(\textbf{z}_2))(\textbf{x}^*_\textbf{n})| \leq \textbf{1}/\textbf{2} & \textnormal{ for }\textbf{1}\leq \textbf{n}\leq \textbf{n}_2.
\end{array}
\]
Now, move on to $\spa_{\textbf{R}}[\textbf{x}^{**},\textbf{x}^{**}-\textbf{j}(\textbf{z}_1),\textbf{x}^{**}-\textbf{j}(\textbf{z}_2)]$. Being it a conditionally finitely generated vector subspace,  we can find $[\textbf{x}^*_\textbf{n} \:;\: \textbf{n}_2<\textbf{n}\leq \textbf{n}_3]\sqsubset \textbf{S}_{\textbf{E}^*}$ such that 
\[
\frac{\Vert \textbf{y}^{**} \Vert}{\textbf{2}} \leq \nmax \left[ |\textbf{y}^{**}(\textbf{x}_\textbf{n})| \:;\: \textbf{n}_2< \textbf{n}\leq \textbf{n}_3 \right] \leq \nmax \left[ |\textbf{y}^{**}(\textbf{x}_\textbf{n})| \:;\: \textbf{1}\leq \textbf{n}\leq \textbf{n}_3 \right]
\]	 	
for any $\textbf{y}^{**}\in \spa_{\textbf{R}}[\textbf{x}^{**},\textbf{x}^{**}-\textbf{j}(\textbf{z}_1),\textbf{x}^{**}-\textbf{j}(\textbf{z}_2)]$.

By recurrence, we obtain two sequences, namely $\textbf{n}_1<\textbf{n}_2,...$ and $\textbf{z}_1,\textbf{z}_2,...$, and a conditional sequence $\{\textbf{x}^*_\textbf{n}\}$ such that for each $k\in\N$ it holds that
\[
\begin{array}{cc}
|(\textbf{x}^{**}-\textbf{j}(\textbf{z}_k))(\textbf{x}^*_\textbf{n})| \leq \textbf{1}/\textbf{k} & \textnormal{ for }\textbf{1}\leq \textbf{n}\leq \textbf{n}_k,
\end{array}
\]
\[
\frac{\Vert \textbf{y}^{**} \Vert}{\textbf{2}} \leq \max \left[ |\textbf{y}^{**}(\textbf{x}_\textbf{n})| \:;\: \textbf{1}\leq \textbf{n}\leq \textbf{n}_k \right].
\]
for any $\textbf{y}^{**}\in \spa_{\textbf{R}}[\textbf{x}^{**},\textbf{x}^{**}-\textbf{j}(\textbf{z}_1),...,\textbf{x}^{**}-\textbf{j}(\textbf{z}_k)]$.

Now, for $\textbf{n}\in\textbf{N}$ with $n:=\sum n_i|a_i$ define $z_n:=\sum z_{n_i}|a_i$. Then $\{\textbf{z}_\textbf{n}\}_{\textbf{n}\in\textbf{N}}$ is a conditional sequence in $\textbf{K}$. Therefore it can be verified that
\begin{equation}
\label{eq I}
\frac{\Vert \textbf{y}^{**} \Vert}{\textbf{2}} \leq \nsup \left[ |\textbf{y}^{**}(\textbf{x}^*_\textbf{n})| \:;\:\textbf{n}\in\textbf{N}\right]
\end{equation}       
for any $\textbf{y}^{**}\in\spa_{\textbf{R}}[\textbf{x}^{**},\textbf{x}^{**}-\textbf{j}(\textbf{z}_1),\textbf{x}^{**}-\textbf{j}(\textbf{z}_2),...]$.

Further, for $\textbf{k}\in\textbf{N}$ with $k:=\sum k_j|b_j\in\textbf{N}$, define $n_k:=\sum n_{k_j}|b_j\in{N}$. Then, for each $\textbf{k}\in\textbf{N}$, it holds

\begin{equation}
\label{eq II}
\begin{array}{cc}
|(\textbf{x}^{**}-\textbf{j}(\textbf{z}_\textbf{k}))(\textbf{x}_\textbf{n}^*)|\leq {\textbf{k}}^{-1}&\textnormal{ for }\textbf{1}\leq\textbf{n}\leq\textbf{n}_\textbf{k}.
\end{array}
\end{equation}  

On the other hand, given that $\textbf{K}$ is conditionally weakly sequentially compact, we have a conditional cluster point $\textbf{x}$ of $\{\textbf{z}_\textbf{k}\}$. It follows that $\textbf{x}\in \overline{\spa_{\textbf{R}}[\textbf{z}_\textbf{k}\:;\: \textbf{k}\in\textbf{N}]}^{\sigma(\textbf{E},\textbf{E}^*)}=\overline{\spa_{\textbf{R}}[\textbf{z}_\textbf{k}\:;\: \textbf{k}\in\textbf{N}]}^{\Vert\cdot\Vert}$, and therefore $\textbf{x}^{**} -\textbf{j}(\textbf{x})\in\overline{\spa_{\textbf{R}}[\textbf{x}^{**},\textbf{x}^{**}-\textbf{j}(\textbf{z}_1),\textbf{x}^{**}-\textbf{j}(\textbf{z}_2),...]}^{\Vert\cdot\Vert}$. 

Now, we will show that (\ref{eq I}) extends to any
\[
\textbf{y}^{**}\in\overline{\spa_{\textbf{R}}[\textbf{x}^{**},\textbf{x}^{**}-\textbf{j}(\textbf{z}_1),\textbf{x}^{**}-\textbf{j}(\textbf{z}_2),... ]}^{\Vert\cdot\Vert}.
\]

Indeed, let us put $\textbf{A}:=\spa_{\textbf{R}}[\textbf{x}^{**},\textbf{x}^{**}-\textbf{j}(\textbf{z}_1),\textbf{x}^{**}-\textbf{j}(\textbf{z}_2),... ]$, and let us suppose that we can find $\textbf{y}^{**}\in\overline{\textbf{A}}^{\Vert\cdot\Vert}$ so that 
\[
\frac{\Vert\textbf{y}^{**}\Vert}{\textbf{2}}|a> \nsup\left[|\textbf{y}^{**}(\textbf{x}^*_\textbf{n})|\:;\:\textbf{n}\in\textbf{N}\right]|a
\]
for some $a\in\A$, $a\neq 0$. We can assume w.l.g. $a=1$. Then, let us choose $\textbf{r}\in\textbf{R}$ such that
\[
\frac{\Vert\textbf{y}^{**}\Vert}{\textbf{2}}>\textbf{r}>\nsup\left[|\textbf{y}^{**}(\textbf{x}^*_\textbf{n})|\:;\:\textbf{n}\in\textbf{N}\right].
\]
We can find a conditional sequence $\{\textbf{y}_\textbf{k}^{**}\}$ in $\textbf{A}$, which conditionally $\Vert\cdot\Vert$-converges to $\textbf{y}^{**}$.  Let us choose $\textbf{k}\in\textbf{N}$ such that 
\[
\Vert\textbf{y}^{**}-\textbf{y}^{**}_\textbf{k}\Vert<(\textbf{2}\textbf{r}-\Vert\textbf{y}^{**}\Vert)\wedge\left(\textbf{r}-\nsup\left[|\textbf{y}^{**}(\textbf{x}^*_\textbf{n})|\:;\:\textbf{n}\in\textbf{N}\right]\right). 
\]
Then, it is not difficult to show that  $|\textbf{y}^{**}_\textbf{k}(\textbf{x}_\textbf{n})|<\textbf{r}<\frac{\Vert\textbf{y}^{**}_\textbf{k}\Vert}{\textbf{2}}$  for any $\textbf{n}\in\textbf{N}$. But this is a contradiction, in view of (\ref{eq I}).

Since (\ref{eq I}) is true for every $\textbf{y}^{**}\in\overline{\textbf{A}}^{\Vert\cdot\Vert}$, in particular it holds for $\textbf{x}^{**}-\textbf{j}(\textbf{x})$, i.e. 
\begin{equation}
\label{eqIII}
\frac{\textbf{1}}{\textbf{2}}\Vert \textbf{x}^{**}-\textbf{j}(\textbf{x})\Vert\leq \nsup\left[|\textbf{x}^{**}(\textbf{x}^*_\textbf{n})-\textbf{x}^*_\textbf{n}(\textbf{x})|\:;\:\textbf{n}\in\textbf{N}\right]. 
\end{equation}

For given $\textbf{p}\leq \textbf{k}$ and $\textbf{n}\leq \textbf{n}_\textbf{p}$, and due to (\ref{eq II}), we have that  
\[
|\textbf{x}^{**}(\textbf{x}^*_\textbf{n})-\textbf{x}^*_\textbf{n}(\textbf{x})| \leq|(\textbf{x}^{**}-\textbf{j}(\textbf{z}_\textbf{k}))(\textbf{x}^*_\textbf{n})|+|\textbf{x}^*_\textbf{n}(\textbf{z}_\textbf{k}) - \textbf{x}^*_\textbf{n}(\textbf{x})| \leq {\textbf{p}}^{-1} +|\textbf{x}^*_\textbf{n}(\textbf{z}_\textbf{k}) - \textbf{x}^*_\textbf{n}(\textbf{x})|.
\] 
Since $\textbf{x}$ is a conditional $\sigma(\textbf{E},\textbf{E}^*)$-cluster point of $\{\textbf{z}_\textbf{k}\}$, we can choose $\textbf{k}\geq \textbf{p}$ so that $|\textbf{x}^*_\textbf{n}(\textbf{z}_\textbf{k}) - \textbf{x}^*_\textbf{n}(\textbf{x})|\leq\frac{\textbf{1}}{\textbf{p}}$. Since $\textbf{p}\in\textbf{N}$ is arbitrary, we obtain that $|\textbf{x}^{**}(\textbf{x}^*_\textbf{n})-\textbf{x}^*_\textbf{n}(\textbf{x})|=\textbf{0}$ for all $\textbf{n}\in\textbf{N}$, and in view of (\ref{eqIII}), we conclude that  $\textbf{x}^{**}=\textbf{j}(\textbf{x})$, which completes the proof.
\end{proof}
  
Now, let us turn to prove Theorem \ref{thm: AmirLindenstrauss}. We need some preliminary results:

\begin{lem}
\label{lem: denseMatch}
Let $\textbf{E}$ be conditionally Banach. Suppose that $\textbf{D}\sqsubset\textbf{E}^*$ is conditionally dense. If $\textbf{K}\sqsubset\textbf{E}$ is conditionally bounded, then $\sigma(\textbf{E},\textbf{E}^*)$ and $\sigma(\textbf{E},\textbf{D})$ agree on $\textbf{K}$.  
\end{lem}
\begin{proof}
It suffices to show that $\textbf{Id}:(\textbf{K},\sigma(\textbf{E},\textbf{E}^*)|_\textbf{K})\rightarrow (\textbf{K},\sigma(\textbf{E},\textbf{D})|_\textbf{K})$ 
is a conditional homeomorphism. It is clear that it is conditionally continuous. To show that the conditional inverse is continuous, take a conditional net $\{\textbf{x}_\textbf{i}\}$ in $\textbf{K}$ such that $\sigma(\textbf{E},\textbf{D})-\nlim_\textbf{i} \textbf{x}_\textbf{i}=\textbf{x}\in\textbf{K}$, and let us show that $\sigma(\textbf{E},\textbf{E}^*)-\nlim_\textbf{i} \textbf{x}_\textbf{i}=\textbf{x}$ too.  Given a conditionally finite subset $\textbf{G}\sqsubset\textbf{E}^*$ on $1$ and $\textbf{r}\in\textbf{R}^{++}$, consider the basic conditional neighborhood $\textbf{U}_{\textbf{G},\textbf{r}}$. Suppose that $\textbf{G}=[\textbf{x}^*_\textbf{n}\:;\:\textbf{n}\leq\textbf{m}]$ and fix $\textbf{s}\in \textbf{R}^{++}$. Then for each $n\in\N$, let $\textbf{y}^*_n\in \textbf{D}$ be such that $\Vert \textbf{x}^*_\textbf{n}-\textbf{y}^*_n\Vert\leq\textbf{s}$. For each $\textbf{n}\in\textbf{N}$ with $n=\sum n_i|a_i$, define $y^*_n:=\sum y^*_n|a_i$ and put $\textbf{H}:=[\textbf{y}^*_\textbf{n}\:;\:\textbf{n}\leq\textbf{m}]$. Then there is $\textbf{i}'$ such that $\textbf{x}_\textbf{i}\in\textbf{x}+\textbf{U}_{\textbf{H},\textbf{r/2}}$ for all $\textbf{i}\geq \textbf{i}'$. Thus, for $\textbf{i}\geq \textbf{i}'$ and $\textbf{n}\leq \textbf{m}$ 
\[
|\textbf{x}^*_\textbf{n}(\textbf{x}_\textbf{i}-\textbf{x})|\leq |(\textbf{x}^*_\textbf{n}-\textbf{y}^*_\textbf{n})(\textbf{x}_\textbf{i})|+|\textbf{y}^*_\textbf{n}(\textbf{x}_\textbf{i}-\textbf{x})|+|(\textbf{y}^*_\textbf{n}-\textbf{x}_\textbf{n}^*)(\textbf{x})|\leq 
\]
\[
\textbf{r}/\textbf{2} + \textbf{2}\textbf{s}\textbf{t}, 
\]
where $\textbf{t}>\textbf{0}$ is chosen so that $\textbf{t}>\Vert\textbf{x}\Vert$ for each $\textbf{x}\in\textbf{K}$.
We are done by just taking $\textbf{s}:=\frac{\textbf{r}}{\textbf{4}\textbf{t}}$.
\end{proof}

\begin{prop}
\label{prop: seqCompactClass}
The class of conditional Banach spaces having conditionally weakly-$*$ sequentially compact dual unit ball is closed under the operation of taking conditionally dense continuous linear images.
\end{prop}

\begin{proof}
For $\textbf{T}\in \textbf{L}(\textbf{E},\textbf{F})$ with $\overline{\textbf{T}(\textbf{E})}^{\Vert\cdot\Vert}=\textbf{F}$, it can be verified by inspection that $\textbf{T}^*:\textbf{F}^*\rightarrow \textbf{E}^*$ defined by $\textbf{T}^*(\textbf{y}^*)(\textbf{x}):=\textbf{y}^*(\textbf{T}(\textbf{x}))$ is a conditionally bounded operator, i.e. $\textbf{T}^*\in \textbf{L}(\textbf{F}^*,\textbf{E}^*)$

Also, it is conditionally injective. Indeed, let us suppose $\textbf{T}^*(\textbf{y}^*)\equiv \textbf{0}$; that is, $\textbf{y}^*(\textbf{T}(\textbf{x}))=\textbf{0}$ for all $\textbf{x}\in\textbf{E}$. We will show that $\textbf{y}^*\equiv \textbf{0}$. For given $\textbf{z}\in \textbf{F}$,  since $\overline{\textbf{T}(\textbf{E})}^{\Vert\cdot\Vert}=\textbf{F}$, there exists a conditional sequence $\{\textbf{x}_\textbf{n}\}$ in $\textbf{E}$ such that $\textbf{T}(\textbf{x}_\textbf{n})$ conditionally converges to $\textbf{z}$. Then $\textbf{y}^*(\textbf{z})=\nlim \textbf{y}^*(\textbf{T}(\textbf{x}_\textbf{n}))=\textbf{0}$.

It follows that $\textbf{T}^*:(\textbf{B}_{\textbf{F}^*},\sigma(\textbf{F}^*,\textbf{T}(\textbf{E}))|_{\textbf{B}_{\textbf{F}^*}})\rightarrow (\textbf{T}^*(\textbf{B}_{\textbf{F}^*}),\sigma(\textbf{E}^*,\textbf{E})|_{\textbf{T}^*(\textbf{B}_{\textbf{F}^*})})$ is a conditional homeomorphism. In fact, it is easy to see that $\textbf{T}^*$  induces a conditional bijection between basic conditional neighborhoods of $\sigma(\textbf{F}^*,\textbf{T}(\textbf{E}))|_{\textbf{B}_\textbf{F}^*}$ and basic conditional neighborhoods of  $\sigma(\textbf{E}^*,\textbf{E})|_{\textbf{T}^*(\textbf{B}_{\textbf{F}^*})}$. Since $\textbf{T}(\textbf{E})\sqsubset\textbf{F}$ is conditionally dense, due to Lemma \ref{lem: denseMatch}, $\textbf{T}^*$ is a conditionally weak-$*$ homeomorphism between $\textbf{B}_{\textbf{F}^*}$ and $\textbf{T}^*(\textbf{B}_{\textbf{F}^*})$.  
  
\end{proof}  

The following lemma is based on a result of A. Grothendieck:

\begin{lem} 
\label{lem: Groth}
Let $\textbf{E}$ be a conditional Bananch space, and let $\textbf{K}\sqsubset \textbf{E}$ be on $1$ and conditionally weakly closed.  Suppose that for each $\textbf{r}\in\textbf{R}^{++}$ there is a conditionally weakly compact set $\textbf{K}_\textbf{r}\sqsubset \textbf{E}$ such that
\[
\textbf{K}\sqsubset \textbf{K}_\textbf{r} + \textbf{r} \textbf{B}_{\textbf{E}}.
\]
Then $\textbf{K}$ is conditionally $\sigma(\textbf{E},\textbf{E}^*)$-compact.
\end{lem}
\begin{proof}

 Let $\overline{\textbf{K}}^{\omega^*}$ denote the conditional $\sigma(\textbf{E}^{**},\textbf{E}^*)$-closure of $\textbf{j}(\textbf{K})$, with $\textbf{j}$ the natural conditional embedding. Due to Lemma \ref{lem: WSClosure}, it suffices to show that $\overline{\textbf{K}}^{\omega^*}\sqsubset\textbf{E}$. 

For each $\textbf{r}\in\textbf{R}^{++}$ it holds that $\overline{\textbf{K}_\textbf{r}}^{\omega^*}=\textbf{K}_\textbf{r}$. Further, the conditional continuity of addition yields
\[
\overline{\textbf{K}}^{\omega^*}\sqsubset \overline{\textbf{K}_\textbf{r} + \textbf{r} \textbf{B}_{\textbf{E}}}^{\omega^*}
\sqsubset \overline{\textbf{K}_\textbf{r}}^{\omega^*} +\overline{\textbf{r} \textbf{B}_{\textbf{E}}}^{\omega^*}
\sqsubset \textbf{j}(\textbf{K}_\textbf{r}) + \textbf{r} \textbf{B}_{\textbf{E}^{**}}.
\]
Consequently 
\[
\overline{\textbf{K}}^{\omega^*}\sqsubset\underset{\textbf{r}\in\textbf{R}^{++}}\sqcap \left(\textbf{j}(\textbf{K}_\textbf{r}) + \textbf{r} \textbf{B}_{\textbf{E}^{**}}\right). 
\]
But $\sqcap\left( \textbf{j}(\textbf{K}_\textbf{r}) + \textbf{r} \textbf{B}_{\textbf{E}^{**}}\right)\sqsubset \textbf{j}(\textbf{E})$.

Indeed, let be given $\textbf{x}\in\sqcap\left( \textbf{j}(\textbf{K}_\textbf{r}) + \textbf{r} \textbf{B}_{\textbf{E}^{**}}\right)$. For each $\textbf{r}\in\textbf{R}^{++}$ let $\textbf{A}_\textbf{r}:=\left[(\textbf{y},\textbf{u})\in\textbf{j}(\textbf{K}_\textbf{r})\Join\textbf{B}_{\textbf{E}^{**}}\:;\:\textbf{x}=\textbf{y}+\textbf{r}\textbf{u}\right]$. Thereby, we obtain a conditional family $\{\textbf{A}_\textbf{r}\}_{\textbf{r}\in\textbf{R}^{++}}$ of conditional sets. The conditional axiom of choice (see \cite[Theorem 2.26]{key-7}) allows us to find a conditional family $\{(\textbf{y}_\textbf{r},\textbf{u}_\textbf{r})\}_{\textbf{r}\in\textbf{R}^{++}}$ such that $\textbf{x}=\textbf{y}_\textbf{r}+\textbf{r}\textbf{u}_\textbf{r}$ with $\textbf{y}_\textbf{r}\in\textbf{j}(\textbf{K}_\textbf{r})$ and $\textbf{u}_\textbf{r}\in\textbf{B}_{\textbf{E}^{**}}$ for all $\textbf{r}\in\textbf{R}^{++}$. Then $\{\textbf{y}_\textbf{r}\}_{\textbf{r}\in\textbf{R}^{++}}$ defines a conditional net in $\textbf{j}(\textbf{E})$ which conditionally converges to $\textbf{x}$. Given that $\textbf{E}$ is conditionally complete, we have that $\textbf{j}(\textbf{E})$ is conditionally closed in $\textbf{E}^{**}$, and it follows that $\textbf{x}\in\textbf{j}(\textbf{E})$.   
  
\end{proof}

\begin{lem} 
\label{lem: banachDisk}
Let $\textbf{E}$ be a conditional Banach space, and let $\textbf{C}\sqsubset\textbf{E}$ be on $1$, conditionally bounded,  convex and balanced. If we define $\textbf{E}_\textbf{C}:=\spa_{\textbf{R}} \textbf{C}$ and the conditional gauge functional $\Vert\cdot\Vert_\textbf{C}$, then:
\begin{enumerate}
	\item $\Vert\cdot\Vert_\textbf{C}$ is a conditional norm on $\textbf{E}_\textbf{C}$ that induces a conditional topology which is a finer  than the one induced by $\Vert\cdot\Vert$ on $\textbf{E}_\textbf{C}$;
	\item If $\textbf{C}$ is conditionally sequentially closed (i.e, $\{\textbf{x}_\textbf{n}\}\sqsubset\textbf{C}$ and $\nlim_\textbf{n}\textbf{x}_\textbf{n}=\textbf{x}$ implies $\textbf{x}\in\textbf{C}$), then $(\textbf{E}_\textbf{C},\Vert\cdot\Vert_\textbf{C})$ is conditionally Banach;
	\item If $\textbf{C}$ is a conditional neighborhood of $\textbf{0}\in \textbf{E}$, then $\textbf{E}_\textbf{C}=\textbf{E}$ and, besides, $\Vert\cdot\Vert$ and $\Vert\cdot\Vert_\textbf{C}$ are equivalent conditional norms on $\textbf{E}$.
\end{enumerate}
\end{lem}
\begin{proof}
\begin{enumerate}
	\item Suppose $\textbf{C}$ is conditionally bounded, convex and balanced. First note that $\textbf{E}_\textbf{C}=\underset{\textbf{n}\in\textbf{N}}\sqcup \textbf{n}\textbf{C}$, hence $\textbf{C}$ is conditionally absorbent in $\textbf{E}_\textbf{C}$. By Proposition \ref{prop: gauge}, $\Vert\cdot\Vert_\textbf{C}$ is a conditional seminorm on $\textbf{E}_\textbf{C}$. Let $\textbf{r}\in \textbf{R}^{++}$ be  such that $\Vert\textbf{x}\Vert\leq \textbf{r}$ for all $\textbf{x}\in \textbf{C}$. Fix $\textbf{x}\in \textbf{E}_\textbf{C}$, and let us take a conditional sequence $\{\textbf{s}_\textbf{n}\}\in\textbf{R}^{++}$ which conditionally converges to $\Vert\textbf{x}\Vert_{\textbf{C}}$ and with $\textbf{x}\in \textbf{s}_\textbf{n}\textbf{C}$ for all $\textbf{n}\in\textbf{N}$. Then $\Vert\textbf{x}\Vert\leq \textbf{s}_\textbf{n} \textbf{r}$, and therefore $\Vert\textbf{x}\Vert\leq \textbf{r}\Vert\textbf{x}\Vert_\textbf{C}$. This proves that $\Vert\cdot\Vert_\textbf{C}$ is a conditional norm and $\mathcal{T}_{\Vert\cdot\Vert}\sqsubset\mathcal{T}_{{\Vert\cdot\Vert}_\textbf{C}}$.
\item Suppose that $\textbf{C}$ is conditionally sequentially closed and let $\{\textbf{x}_\textbf{n}\}$ be a conditional Cauchy sequence in $(\textbf{E}_\textbf{C},\Vert\cdot\Vert_\textbf{C})$. Then, we can choose a sequence $\textbf{n}_1<\textbf{n}_2<...$ such that $\textbf{x}_{\textbf{n}_{k+1}}-\textbf{x}_{\textbf{n}_k}\in \textbf{2}^{-\textbf{k}}\textbf{C}$ for each $k\in\N$. For $\textbf{k}\in\textbf{N}$ of the form $k=\sum k_i|a_i$, we define $y_k:=\sum x_{n_{k_i}}|a_i$. By doing so, we obtain a conditional sequence $\{\textbf{y}_\textbf{k}\}$ such that $\textbf{y}_{\textbf{k}+\textbf{1}}-\textbf{y}_\textbf{k}\in \textbf{2}^{-\textbf{k}}\textbf{C}$ for each $\textbf{k}\in\textbf{N}$. Since $\textbf{C}$ is conditionally bounded, it follows that $\{\textbf{y}_\textbf{k}\}$ is also conditionally Cauchy with respect to $\Vert\cdot\Vert$. Now, since
\[
\begin{array}{cc}
\textbf{y}_\textbf{n}=\textbf{y}_\textbf{1} + \underset{\textbf{1}\leq\textbf{k}\leq\textbf{n}}\sum (\textbf{y}_\textbf{k}-\textbf{y}_{\textbf{k}-\textbf{1}})\in \textbf{y}_\textbf{1} + \underset{\textbf{1}\leq\textbf{k}\leq\textbf{n}}\sum \textbf{2}^\textbf{k}\textbf{C}\sqsubset\textbf{y}_\textbf{1} +\textbf{C}& \textnormal{ for all }\textbf{n}\in\textbf{N},
\end{array}
\]

we have that $\{\textbf{y}_\textbf{n}\}$ is in $\textbf{y}_\textbf{1}+\textbf{C}$. Since $\textbf{y}_\textbf{1}+\textbf{C}$ is conditionally sequentially closed, there exists $\textbf{x}\in\textbf{y}_\textbf{1}+\textbf{C}\sqsubset\textbf{E}_\textbf{C}$ such that $\{\textbf{y}_\textbf{n}\}$ conditionally converges to $\textbf{x}$. Let us show that $\Vert \textbf{y}_\textbf{n}-\textbf{x}\Vert_\textbf{C}$ conditionally converges to $\textbf{0}$. For $\textbf{p}\geq \textbf{q}\geq \textbf{2}$, $\textbf{p},\textbf{q}\in\textbf{N}$,
\[
\textbf{y}_\textbf{p}-\textbf{y}_\textbf{q}= \sum_{\textbf{q}+\textbf{1}\leq\textbf{k}\leq\textbf{p}}(\textbf{y}_\textbf{k}-\textbf{y}_{\textbf{k}-\textbf{1}})\in \sum_{\textbf{1}\leq\textbf{k}\leq\textbf{p}} \textbf{2}^{-(\textbf{k}-\textbf{1})}\textbf{C}\sqsubset \textbf{2}^{-(\textbf{q}-\textbf{1})}\textbf{C}.
\]   
We obtain that $\textbf{y}_\textbf{p}-\textbf{y}_\textbf{q}\in \textbf{2}^{-(\textbf{q}-\textbf{1})}\textbf{C}$ whenever $\textbf{p},\textbf{q}\in\textbf{N}$, $\textbf{p}\geq\textbf{q}\geq \textbf{2}$. Being $\textbf{C}$ conditionally sequentially closed, it follows that $\textbf{x}-\textbf{y}_\textbf{q}\in \textbf{2}^{-(\textbf{q}-\textbf{1})}\textbf{C}$ for all $\textbf{q}$, and therefore $\Vert\textbf{x}-\textbf{y}_\textbf{q}\Vert_\textbf{C}$ conditionally converges to $\textbf{0}$. Then, by using that $\{\textbf{x}_\textbf{n}\}$ is conditionally Cauchy, it is easy to prove that it also conditionally converges to $\textbf{x}$. 
\item
 Now suppose that $\textbf{C}$ is a conditional neighborhood of $\textbf{0}\in \textbf{E}$. Let us take a conditional ball $\textbf{r}\textbf{B}_\textbf{E}$ conditionally contained into $\textbf{C}$. Then,  $\textbf{r}\textbf{x}\Vert\textbf{x}\Vert^{-1}\in \textbf{C}$ for all $\textbf{x}\in\textbf{E}$, and consequently $\textbf{r}\Vert\textbf{x}\Vert_\textbf{C} \leq \Vert\textbf{x}\Vert$. This proves that both conditional norms are equivalent.
\end{enumerate}
\end{proof}

\begin{thm}
\label{thm: reflex}
A conditional Banach space $\textbf{E}$ is conditionally reflexive (i.e. $\textbf{j}(\textbf{E})=\textbf{E}^{**}$) if, and only if, the conditional unit ball $\textbf{B}_{\textbf{E}}$ is conditionally weakly compact.

In that case, the conditional dual unit ball $\textbf{B}_{\textbf{E}^*}$ is conditionally weakly-$*$ sequentially compact.  
\end{thm}
\begin{proof}
First suppose that $\textbf{E}$ is conditionally reflexive, then $\textbf{B}_{\textbf{E}}$ is conditionally weakly compact as a consequence of Lemma \ref{lem: WSClosure}.

Conversely, suppose that $\textbf{B}_{\textbf{E}}$ is conditionally weakly compact. Then it is clear that $\textbf{j}(\textbf{B}_{\textbf{E}})$ is conditionally $\sigma(\textbf{E}^{**},\textbf{E}^*)$-compact. Therefore, $\textbf{j}(\textbf{B}_{\textbf{E}})$ is a conditionally $\sigma(\textbf{E}^{**},\textbf{E}^*)$-closed subset of $\textbf{B}_{\textbf{E}^{**}}$. But Theorem \ref{thm: Goldstine} yields that $\textbf{j}(\textbf{B}_{\textbf{E}})=\textbf{B}_{\textbf{E}^{**}}$.

Let us show the second part of the theorem.
By the conditional Banach-Alaoglu theorem (see \cite[Theorem 5.10]{key-7}), $\textbf{B}_{\textbf{E}^*}$ is conditionally weakly-$*$ compact, and as $\textbf{E}$ is conditionally reflexive, we have that $\textbf{B}_{\textbf{E}^*}$ is conditionally $\sigma(\textbf{E}^*,\textbf{E}^{**})$-compact. Besides, by  Theorem \ref{thm: EberleinSmulian}, $\textbf{B}_{\textbf{E}^*}$ is conditionally sequentially $\sigma(\textbf{E}^*,\textbf{E}^{**})$-compact, and by using again that $\textbf{E}$ is conditionally reflexive, we obtain that $\textbf{B}_{\textbf{E}^*}$ is conditionally weak-$*$ sequentially compact.
\end{proof}

\begin{lem}
\label{lem: l2}
Let $\{\textbf{E}_\textbf{n}\}$ be a conditionally countable family of conditional Banach spaces $(\textbf{E}_\textbf{n},\Vert\cdot\Vert_\textbf{n})$. Consider the conditional set
\[
\oplus_{\textbf{n}\in\textbf{N}}^2 \textbf{E}_\textbf{n} :=\left[ \{\textbf{x}_\textbf{n}\}\in \conprod_{\textbf{n}\in\textbf{N}} \textbf{E}_\textbf{n} \:;\: \sum_{\textbf{n}\geq \textbf{1}} \Vert \textbf{x}_\textbf{n}\Vert^\textbf{2}_\textbf{n}\in \textbf{R}\right].
\]
Then we have the following:

\begin{enumerate}
	\item $(\oplus_{\textbf{n}\in\textbf{N}}^2 \textbf{E}_\textbf{n},\Vert\cdot\Vert_2)$ with $\Vert\cdot\Vert_2:=( \sum_{\textbf{n}\geq \textbf{1}} \Vert \textbf{x}_\textbf{n}\Vert^\textbf{2})^\frac{\textbf{1}}{\textbf{2}}$ is a conditional Banach space;
	\item $\left(\oplus_{\textbf{n}\in\textbf{N}}^2 \textbf{E}_\textbf{n}\right)^*=\oplus_{\textbf{n}\in\textbf{N}}^2 \textbf{E}_\textbf{n}^*$;
	\item The conditional set
	\[
\oplus_{\textbf{n}\in\textbf{N}}^0 \textbf{E}_\textbf{n}^*:=\left[(\textbf{x}_\textbf{n}^*) \:;\: \exists \textbf{m}\in\textbf{N},\: \textbf{x}_\textbf{n}=\textbf{0}\textnormal{ for all }\textbf{n}>\textbf{m}\right]
\] 
	is conditionally dense in $\oplus_{\textbf{n}\in\textbf{N}}^2 \textbf{E}_\textbf{n}$.
\end{enumerate}
\end{lem}
\begin{proof}
\begin{enumerate}
	\item 
	Inspection shows that $\oplus_{\textbf{n}\in\textbf{N}}^2 \textbf{E}_\textbf{n}$ is a conditional vectorial space and $\Vert\cdot\Vert$ is a conditional norm.

Let us prove that $\oplus_{\textbf{n}\in\textbf{N}}^2 \textbf{E}_\textbf{n}$ is conditionally complete. Let $\{\textbf{x}_\textbf{n}\}$ be a conditionally Cauchy sequence in $\oplus_{\textbf{n}\in\textbf{N}}^2 \textbf{E}_\textbf{n}$, with $\textbf{x}_\textbf{n}=\{\textbf{x}^\textbf{m}_\textbf{n}\}_{\textbf{m}\in\textbf{N}}$ for each $\textbf{n}\in\textbf{N}$. Let be given $\textbf{r}>\textbf{0}$, let us choose an $\textbf{n}_\textbf{r}\in\textbf{N}$ such that
\[
\begin{array}{cc}
\Vert\textbf{x}_\textbf{p} - \textbf{x}_\textbf{q}\Vert_2^\textbf{2}= \sum_{\textbf{m}\geq\textbf{1}} \Vert \textbf{x}^\textbf{m}_\textbf{p}-\textbf{x}^\textbf{m}_\textbf{q}\Vert^\textbf{2}_\textbf{m}\leq \textbf{r}^\textbf{2}, & \textnormal{ for all }\textbf{p}\geq\textbf{q}\geq\textbf{n}_\textbf{r}.
\end{array}
\] 
We see that, for each $\textbf{m}$, the conditional sequence $\{\textbf{x}^\textbf{m}_\textbf{n}\}_{\textbf{n}\in\textbf{N}}$ is conditionally Cauchy and, being $\textbf{E}_\textbf{m}$ conditionally Banach, it converges to some $\textbf{x}^\textbf{m}\in \textbf{E}_\textbf{m}$. Besides, note that the application $m\mapsto x^m$ is stable, hence $\textbf{x}:=\{\textbf{x}^\textbf{n}\}\in \conprod_{\textbf{n}\in\textbf{N}} \textbf{E}_\textbf{n}$.

Now, for fixed $\textbf{k}\in\textbf{N}$ and for each $\textbf{n}\geq \textbf{n}_\textbf{r}$, we have 
\[
\sum_{\textbf{1}\leq\textbf{m}\leq\textbf{k}}\Vert \textbf{x}^\textbf{m}- \textbf{x}^\textbf{m}_\textbf{n} \Vert_\textbf{m}^\textbf{2}=\nlim_{\textbf{j}} \sum_{\textbf{1}\leq\textbf{m}\leq\textbf{k}}\Vert \textbf{x}^\textbf{m}_\textbf{j}- \textbf{x}^\textbf{m}_\textbf{n} \Vert_\textbf{m}^\textbf{2}\leq \textbf{r}^\textbf{2}.
\]
Then, by taking conditional limits in $\textbf{k}$, we obtain $\Vert \textbf{x}-\textbf{x}_\textbf{n} \Vert_2\leq \textbf{r}$ for $\textbf{n}\geq \textbf{n}_\textbf{r}$. This means that, first,  $\textbf{x}-\textbf{x}_\textbf{n}\in \oplus_{\textbf{n}\in\textbf{N}}^2 \textbf{E}_\textbf{n}$ for $\textbf{n}\geq \textbf{n}_\textbf{r}$, and consequently so does $\textbf{x}$; second, since $\textbf{r}>\textbf{0}$ is arbitrary, it follows that $\nlim \textbf{x}_\textbf{n}=\textbf{x}$. 

\item Consider the conditional function $\textbf{T}:\oplus_{\textbf{n}\in\textbf{N}}^2 \textbf{E}_\textbf{n}^*\rightarrow\left(\oplus_{\textbf{n}\in\textbf{N}}^2 \textbf{E}_\textbf{n}\right)^*$, $\textbf{x}^*\mapsto\textbf{T}_{\textbf{x}^*}$, where 
\[
\begin{array}{cc}
\textbf{T}_{\textbf{x}^*}(\{\textbf{x}_\textbf{n}\}):=\sum_{\textbf{n}\geq\textbf{1}}\textbf{x}^*_\textbf{n}(\textbf{x}_\textbf{n}), & \textnormal{ with }\textbf{x}^*=\{\textbf{x}^*_\textbf{n}\}.
\end{array}
\]

We claim that $\textbf{T}$ is a conditional isometric isomorphism.

The Cauchy-Schwarz inequality for conditional series, which is  proved in the Appendix (see Proposition \ref{prop: CSinequality}), allows us to show that $\textbf{T}$ is well defined and 
\begin{equation}
\label{eq: kk}
\Vert\textbf{T}_{\textbf{x}^*}\Vert\leq \Vert\textbf{x}^*\Vert_2.
\end{equation}
On the other hand, let be given $\textbf{x}^*=\{\textbf{x}^*_\textbf{n}\}\in \oplus_{\textbf{n}\in\textbf{N}}^2 \textbf{E}_\textbf{n}^*$. 

For each $n\in\N$, we can choose $\textbf{y}_n\in\textbf{E}_\textbf{n}$ with $\Vert\textbf{y}_n\Vert_\textbf{n}=\textbf{1}$. For $\textbf{n}\in\textbf{N}$ of the form $n=\sum n_i|a_i$, let us put $\textbf{y}_\textbf{n}=y_n|1$ with ${y}_{n}:=\sum y_{n_i}|a_i$. Doing so, we obtain a conditional sequence $\{\textbf{y}_\textbf{n}\}$ so that $\Vert\textbf{y}_\textbf{n}\Vert_\textbf{n}=\textbf{1}$ for each $\textbf{n}$.

Now, we construct the conditional sequence $\{\textbf{x}_\textbf{n}\}$ where $\textbf{x}_\textbf{n}:=\textbf{x}^*_\textbf{n}(\textbf{y}_\textbf{n})\textbf{y}_\textbf{n}$.

Note that $\Vert\textbf{x}_\textbf{n}\Vert_\textbf{n}=|\textbf{x}^*_\textbf{n}(\textbf{y}_\textbf{n})|$ for each $\textbf{n}$. Consequently we have  
\begin{equation}
\label{eq: l2}
\begin{array}{cc}
\textbf{x}^*_\textbf{n}(\textbf{x}_\textbf{n})=\Vert\textbf{x}_\textbf{n}\Vert^\textbf{2}_\textbf{n}, & \textnormal{ for each }\textbf{n}\in\textbf{N}.
\end{array}
\end{equation}

 Now, for given $\textbf{n},\textbf{m}\in\textbf{N}$, let us define 
\[
\begin{array}{cc}
\textbf{z}^\textbf{m}_\textbf{n}:=\textbf{x}_\textbf{n}|a + \textbf{0}|a^c, &\textnormal{ with }a:=\vee\{b\in\A\:;\:\textbf{n}|b\leq\textbf{m}|b\}.
\end{array}
\] 

For $\textbf{m}\in\textbf{N}$, we have that $\textbf{z}^\textbf{m}:=\{\textbf{z}^\textbf{m}_\textbf{n}\}_{\textbf{n}\in\textbf{N}}$ is a conditional sequence.

Then, in view of (\ref{eq: l2}), for fixed $\textbf{m}\in\textbf{N}$, it follows that
\[
\textbf{T}_{\textbf{x}^*}(\textbf{z}^\textbf{m})=\sum_{\textbf{1}\leq\textbf{n}\leq\textbf{m}}\textbf{x}^*_\textbf{n}(\textbf{x}_\textbf{n})=\sum_{\textbf{1}\leq\textbf{n}\leq\textbf{m}} \Vert\textbf{x}_\textbf{n}\Vert_\textbf{n}^\textbf{2}.
\] 

The inequality $\left|\textbf{T}_{\textbf{x}^*}(\textbf{z}^\textbf{m})\right|\leq\Vert\textbf{T}_{\textbf{x}^*}\Vert\Vert \textbf{z}^\textbf{m}\Vert_2$ together with the above equation, implies that
\[
\sum_{\textbf{1}\leq\textbf{n}\leq\textbf{m}} \Vert\textbf{x}_\textbf{n}\Vert_\textbf{n}^\textbf{2}\leq \Vert\textbf{T}_{\textbf{x}^*}\Vert \Vert\textbf{z}^\textbf{m}\Vert_2 = \Vert\textbf{T}_{\textbf{x}^*}\Vert \left(\sum_{\textbf{1}\leq\textbf{n}\leq\textbf{m}}\Vert\textbf{x}\Vert^\textbf{2}_\textbf{n}  \right)^\frac{\textbf{1}}{\textbf{2}}.
\]
It follows that
\[
\Vert\textbf{T}_{\textbf{x}^*}\Vert\geq \left(\sum_{\textbf{1}\leq\textbf{n}\leq\textbf{m}}\Vert\textbf{x}^*_\textbf{n}\Vert^\textbf{2}_\textbf{n}  \right)^\frac{\textbf{1}}{\textbf{2}}.
\]
Since $\textbf{m}\in\textbf{N}$ is arbitrary, we finally obtain $\Vert\textbf{T}_{\textbf{x}^*}\Vert\geq \Vert\textbf{x}^*\Vert_2$.

Let us show that $\textbf{T}$ is conditionally surjective. Indeed, suppose $\textbf{x}^*\in\left(\oplus_{\textbf{n}\in\textbf{N}}^2 \textbf{E}_\textbf{n}\right)^*$. For each $\textbf{n}\in\textbf{N}$, we define $\textbf{x}^*_\textbf{n}:\textbf{E}_\textbf{n}\rightarrow\textbf{R}$, $\textbf{x}^*_\textbf{n}(\textbf{x}):=\textbf{x}^*(\{\textbf{y}_\textbf{m}\}_{\textbf{m}\in\textbf{N}})$ with 
\[
\begin{array}{cc}
{y}_{m}:=x_n|a + {0}|a^c, & a:=\vee\{b\in\A\:;\:\textbf{m}|b=\textbf{n}|b\}.
\end{array}
\]
By doing so, we obtain
\[
|\textbf{x}^*_\textbf{n}(\textbf{x})|=|\textbf{x}^*(\{\textbf{y}_\textbf{m}\})|\leq\Vert\textbf{x}^*\Vert\Vert\{\textbf{y}_\textbf{m}\}\Vert_2=\Vert\textbf{x}^*\Vert\Vert\textbf{x}\Vert_\textbf{n}.
\]
This mean that $\textbf{x}^*_\textbf{n}\in \textbf{E}_\textbf{n}$ for each $\textbf{n}$. Besides, $\textbf{T}(\{\textbf{x}_\textbf{n}^*\})=\textbf{x}^*$.

\item Given $\textbf{x}=\{\textbf{x}_\textbf{n}\}\in\oplus_{\textbf{n}\in\textbf{N}}^2 \textbf{E}_\textbf{n}$, we take the conditional sequence $\{\textbf{s}_\textbf{n}\}$ where $\textbf{s}_\textbf{n}:=\sum_{\textbf{1}\leq\textbf{k}\leq\textbf{n}}\textbf{x}_\textbf{k}$. Inspection shows that this conditional sequence conditionally converges to $\textbf{x}$.
\end{enumerate}
\end{proof}

\begin{lem}
\label{lem: Amir-Lin}
Let $(\textbf{E},\Vert\cdot\Vert)$ be a conditional Bananch space, and let $\textbf{K}\sqsubset \textbf{E}$ be on $1$, conditionally weakly compact, convex and balanced. Then there is a conditional subset $\textbf{C}$ of $\textbf{E}$ which is also conditionally weakly compact, convex and  balanced with $\textbf{K}\sqsubset\textbf{C}$, and so that $(\textbf{E}_\textbf{C},\Vert\cdot\Vert)$ is a conditionally reflexive normed space.   
\end{lem}
\begin{proof}
For each $\textbf{n}\in\textbf{N}$ put $\textbf{B}_{\textbf{n}}:=\left(\textbf{2}^\textbf{n}\textbf{K}+\textbf{2}^{-\textbf{n}} \textbf{B}_\textbf{E}\right)$. Since $\textbf{K}$ is conditionally closed, convex and balanced, $\textbf{B}_\textbf{n}$ is conditionally weakly closed, convex and balanced. Further, $\textbf{B}_\textbf{n}$ is a conditional neighborhood of $\textbf{0}\in \textbf{E}$ since it conditionally contains $\textbf{2}^{-\textbf{n}} \textbf{B}_\textbf{E}$. By Lemma \ref{lem: banachDisk}, $\textbf{E}$ can be conditionally renormed obtaining an equivalent conditional norm $\Vert\cdot\Vert_\textbf{n}$ on $\textbf{E}$. Let $\textbf{C}:=[\textbf{x}\in\textbf{E}\:;\:\sum_{\textbf{n}\geq\textbf{1}}\Vert \textbf{x}\Vert_\textbf{n}^\textbf{2}\leq \textbf{1} ]$. 

Then 
\[
\textbf{C}=\underset{\textbf{n}\in\textbf{N}}\sqcap\left[\textbf{x}\in\textbf{E}\:;\:\sum_{\textbf{1}\leq\textbf{k}\leq\textbf{n}}\Vert\textbf{x}\Vert_\textbf{k}^\textbf{2}\leq \textbf{1}\right]. 
\]
So $\textbf{C}$ is conditionally norm closed, convex and balanced; hence, by Proposition \ref{prop: weakClosure}, it is also conditionally weakly closed. It can be easily verified that $\Vert \textbf{x}\Vert_{\textbf{C}}=\left(\sum_{\textbf{n}\geq \textbf{1}}\Vert \textbf{x}\Vert_\textbf{n}^\textbf{2}\right)^{\textbf{1}/\textbf{2}}$ for $\textbf{x}\in\textbf{E}_\textbf{C}$.   
Then:
\begin{enumerate}
	\item For $\textbf{x}\in\textbf{K}$ it holds that $\textbf{2}^\textbf{n}\textbf{x}\in \textbf{2}^\textbf{n}\textbf{K}$ which is conditionally contained in $\textbf{B}_\textbf{n}$, and therefore, $\Vert \textbf{x}\Vert_\textbf{n}\leq \textbf{2}^{-\textbf{n}}$. Then $\Vert\textbf{x}\Vert_{\textbf{C}}=\sum_{\textbf{n}\geq \textbf{1}} \Vert \textbf{x}\Vert_\textbf{n}^\textbf{2}\leq\sum_{\textbf{n}\geq\textbf{1}} \textbf{2}^{-\textbf{2} \textbf{n}}=: \textbf{r}\leq \textbf{1}$. So $\textbf{K}\sqsubset\textbf{C}.$
	\item $\textbf{C}$ is conditional weakly compact in $\textbf{E}$. This follows from Lemma \ref{lem: Groth} by just noticing that $\textbf{C}$ is conditionally weakly closed and $\textbf{C}\sqsubset \textbf{2}^\textbf{n}+ \textbf{2}^{-\textbf{n}} \textbf{B}_{\textbf{E}}$ for each $\textbf{n}\in\textbf{N}$.
	\item 
The conditional topologies $\sigma(\textbf{E},\textbf{E}^*)$ and $\sigma(\textbf{E}_\textbf{C},\textbf{E}^*_\textbf{C})$ coincide on $\textbf{C}$. Note that the conditional function $\textbf{E}_\textbf{C}\rightarrow  \oplus_{\textbf{n}\in\textbf{N}}^2 \textbf{E}_\textbf{n}$ which sends $\textbf{x}$ to $\{\textbf{y}_\textbf{n}\}$ with $\textbf{y}_\textbf{n}=\textbf{x}$ for each $\textbf{n}\in\textbf{N}$, is a conditional isometric embedding.

Besides, in view of Lemma \ref{lem: l2}, we have that  $(\oplus_{\textbf{n}\in\textbf{N}}^2 \textbf{E}_\textbf{n})^*=\oplus_{\textbf{n}\in\textbf{N}}^2 \textbf{E}_\textbf{n}^*$, and $\oplus_{\textbf{n}\in\textbf{N}}^0 \textbf{E}_\textbf{n}^*$ is a conditionally dense subset of $\oplus_{\textbf{n}\in\textbf{N}}^2 \textbf{E}_\textbf{n}$.

Therefore, due to Lemma \ref{lem: denseMatch},  we have that the conditional topologies 
\[
\begin{array}{ccc}
\sigma(\oplus_{\textbf{n}\in\textbf{N}}^2 \textbf{E}_\textbf{n},(\oplus_{\textbf{n}\in\textbf{N}}^2 \textbf{E}_\textbf{n})^*) & \textnormal{ and } & \sigma(\oplus_{\textbf{n}\in\textbf{N}}^2 \textbf{E}_\textbf{n},\oplus_{\textbf{n}\in\textbf{N}}^0 \textbf{E}_\textbf{n}^*)
\end{array}
\]
 agree on the conditionally bounded subset $\textbf{C}$.

But on $\textbf{C}$ the conditional topology $\sigma(\oplus_{\textbf{n}\in\textbf{N}}^2 \textbf{E}_\textbf{n},\oplus_{\textbf{n}\in\textbf{N}}^0 \textbf{E}_\textbf{n}^*)$ turns out to be the conditional weak topology  $\sigma(\textbf{E},\textbf{E}^*)$.
\end{enumerate}
We conclude that, as $\textbf{C}$ is conditionally $\sigma(\textbf{E}_\textbf{C},\textbf{E}^*_\textbf{C})$-compact, $\textbf{E}_\textbf{C}$ is conditionally reflexive in view of Theorem \ref{thm: reflex}.

\end{proof}

Finally let us turn out to prove Theorem \ref{thm: AmirLindenstrauss}

\begin{proof}
Suppose that $\textbf{E}=\overline{\spa_{\textbf{R}}\textbf{K}}$ where $\textbf{K}\sqsubset\textbf{E}$ is on $1$, conditionally weakly compact, convex and balanced. Let \textbf{C} be the conditionally  weakly compact, convex and balanced subset provided by Lemma \ref{lem: Amir-Lin}. On one hand, we have that the conditional function $\textbf{i}:\textbf{E}_\textbf{C}\rightarrow\textbf{E}$, $\textbf{x}\mapsto\textbf{x}$  is conditionally linear continuous with conditionally dense range. On the other hand, second part of Theorem \ref{thm: reflex} tells us that, being $\textbf{E}_\textbf{C}$ conditionally reflexive, it has conditionally weakly-$*$ sequentially compact dual unit balls. Hence, by Proposition \ref{prop: seqCompactClass}, we conclude that the same occurs for $\textbf{E}$. 
\end{proof}

\section{Appendix}

T. Guo et al. \cite{key-2} provided an study of the algebraic structure of finitely generated $L^0$-modules. Also P. Cheridito et al. \cite{key-19} the notion of linear $L ^0$-independence in the module $(L^0)^d$. 

Then we have the following results, which can be proved with techniques which are similar employed in \cite[Section 2]{key-19}. For the sake of saving space we omit the proof:

\begin{prop}
If $\textbf{E}$ is a conditionally finitely generated vector space, then there exists a conditionally finite subset $[\textbf{x}_\textbf{k}\:;\:\textbf{1}\leq\textbf{k}\leq\textbf{n}]$ of $\textbf{E}$ such that
\begin{enumerate}
	\item $\textbf{E}=\spa_\textbf{R}[\textbf{x}_\textbf{k}\:;\:\textbf{1}\leq\textbf{k}\leq\textbf{n}]$;
	\item the conditional function $\textbf{T}:\textbf{R}^\textbf{n}\rightarrow\textbf{E}$ given by $\textbf{T}(\{\textbf{x}_{\textbf{1}\leq\textbf{k}\leq\textbf{n}}\}):=\sum_{\textbf{1}\leq\textbf{k}\leq\textbf{n}} \textbf{r}_\textbf{k}\textbf{x}_\textbf{k}$ is a conditionally isomorphism.
\end{enumerate}
Furthermore, this $\textbf{n}$ is unique in the sense of it does not depend on the conditional finite subset chosen.
\end{prop}
 
%

Down below it is stated and proved the conditional version of the Heine-Borel Theorem:

\begin{prop}
\label{prop: Heine-BorelExt}
If $(\textbf{E},\Vert\cdot\Vert)$ is a conditionally finitely generated  normed space then it holds:

\begin{enumerate}
	\item $\textbf{T}$ is a conditional homeomorphism.
	\item Every conditionally closed and bounded subset of $\textbf{E}$ is conditionally compact.
\end{enumerate}
\end{prop}
\begin{proof}
\begin{enumerate}

\item
The conditional continuity of $\textbf{T}$ is a consequence of conditional boundedness. Indeed, for $\textbf{z}\in\textbf{R}^\textbf{n}$
\[
\begin{array}{cc}
\Vert \textbf{T}(\textbf{z})\Vert \leq\left\Vert\sum_{1\leq\textbf{k}\leq\textbf{n}} \textbf{r}_\textbf{k} \textbf{x}_\textbf{k}\right\Vert \leq \textbf{r}\Vert\textbf{z}\Vert_\textbf{1},& \textnormal{ with } \textbf{r}:=\nsup[\Vert\textbf{x}_\textbf{k}\Vert\:;\: \textbf{k}\leq\textbf{n}].  
\end{array}
\]

Let us show that $\textbf{T}^{-1}$ is conditionally continuous. Firstly, note that $\textbf{S}:=[\textbf{z}\in\textbf{R}^\textbf{n}\:;\: \Vert\textbf{z}\Vert_\textbf{1}=\textbf{1}]$, being conditionally closed and bounded, is conditionally compact in $\textbf{R}^\textbf{n}$ by the conditional Heine-Borel theorem (see \cite[Theorem 5.5.8]{key-36}). Also the conditional function $\textbf{f}:\textbf{S}\rightarrow\textbf{R}$ defined by $\textbf{f}(\textbf{z}):=\Vert\textbf{T}(\textbf{z})\Vert$ is conditionally continuous and, since $\textbf{T}$ is conditionally injective, $\textbf{0}\in\textbf{f}(\textbf{S})^\sqsubset$. Consequently, $\textbf{f}$ attains its minimum in $\textbf{S}$, hence $0<\textbf{r}:=\nmin[\textbf{f}(\textbf{z})\:;\:\textbf{z}\in\textbf{S}]$. Then for $\textbf{z}\in\textbf{R}^\textbf{n}$ we have $\textbf{f}(\textbf{z})\geq\textbf{r}\Vert\textbf{z}\Vert_\textbf{1}$, which yields $\textbf{r}\Vert\textbf{z}\Vert_\textbf{1}\leq\Vert\textbf{T}(\textbf{z})\Vert$, and therefore $\textbf{T}^{-1}$ is conditionally continuous.  
\item Let $\textbf{K}$ be a conditionally closed and bounded subset of $\textbf{E}$. Since $\textbf{T}$ is continuous, $\textbf{T}^{-1}(\textbf{K})$ is conditionally closed; since $\textbf{T}^{-1}$ is conditionally linear and continuous $\textbf{T}^{-1}(\textbf{K})$ is conditionally bounded. By the conditional Heine-Borel theorem (see \cite[Theorem 5.5.8]{key-36}), $\textbf{T}^{-1}(\textbf{K})$ is conditionally compact. Since $\textbf{T}$ is continuous, in view of \cite[Proposition 3.26]{key-7}), we have that $\textbf{K}=\textbf{T}(\textbf{T}^{-1}(\textbf{K}))$ is conditionally compact.    

\end{enumerate}

\end{proof}

We have below the conditional version of the Cauchy-Schwarz inequality for conditional series:

\begin{prop}
\label{prop: CSinequality}
For given two conditional sequences $\{\textbf{a}_\textbf{n}\},\:\{\textbf{b}_\textbf{n}\}$ in $\textbf{R}$, we have that 
\begin{equation}
\left(\sum_{\textbf{k}\geq\textbf{1}}{\textbf{a}_\textbf{k}\textbf{b}_\textbf{k}}\right)^\textbf{2}\leq\left(\sum_{\textbf{k}\geq\textbf{1}}{\textbf{a}_\textbf{k}^\textbf{2}}\right)\left(\sum_{\textbf{k}\geq\textbf{1}}{\textbf{b}_\textbf{k}^\textbf{2}}\right).
\end{equation} 
\end{prop}
\begin{proof}
For fixed $\textbf{n}\in\textbf{N}$, it follows that
\[
\begin{array}{cc}
\sum_{\textbf{1}\leq \textbf{k}\leq\textbf{n}}({\textbf{a}_\textbf{k} \textbf{x}+\textbf{b}_\textbf{k}})^\textbf{2}\geq\textbf{0} & \textnormal{ for all }\textbf{x}\in \textbf{R}.
\end{array}
\]
This means that
\begin{equation}
\label{eq: cs2}
\begin{array}{cc}
\textbf{a}\textbf{x}^\textbf{2}+\textbf{2}\textbf{b}\textbf{x}+\textbf{c}\geq\textbf{0} & \textnormal{ for all }\textbf{x}\in \textbf{R},
\end{array}
\end{equation}
where $\textbf{a}=\sum_{\textbf{1}\leq \textbf{k}\leq\textbf{n}}{\textbf{a}^\textbf{2}_\textbf{k}}$, $\textbf{b}=\sum_{\textbf{1}\leq \textbf{k}\leq\textbf{n}}{\textbf{a}_\textbf{k}\textbf{b}_\textbf{k}}$ and $\textbf{c}=\sum_{\textbf{1}\leq \textbf{k}\leq\textbf{n}}{\textbf{b}^\textbf{2}_\textbf{k}}$.

The result is trivial on $\supp(\textbf{a})^c$. So, by consistency, we can suppose w.l.g. $\supp(\textbf{a})=1$.

Doing so, we have that $\textbf{a}\in [\textbf{0}]^{\sqsubset}$. Now, if we take $\textbf{x}:=-\frac{\textbf{b}}{\textbf{a}}$,  from  (\ref{eq: cs2}) we obtain $\textbf{b}^\textbf{2}-\textbf{a}\textbf{c}\leq\textbf{0}$. Equivalently
\[ 
\left(\sum_{\textbf{1}\leq\textbf{k}\leq\textbf{n}}{\textbf{a}_\textbf{k}\textbf{b}_\textbf{k}}\right)^\textbf{2}\leq\left(\sum_{\textbf{1}\leq\textbf{k}\leq\textbf{n}}{\textbf{a}_\textbf{k}^\textbf{2}}\right)\left(\sum_{\textbf{1}\leq\textbf{k}\leq\textbf{n}}{\textbf{b}_\textbf{k}^\textbf{2}}\right).
\]
The result follows, just by taking conditional limit in $\textbf{n}$.
\end{proof}

\vspace{1cc}

Jos\'e M. Zapata. Universidad de Murcia, Dpto. Matem\'aticas, 30100 Espinardo, Murcia, Spain, e-mail: jmzg1@um.es

{\small
\noindent

}\end{document}